\documentclass{compositio}
\usepackage{hyperref}

\usepackage{xypic}

\usepackage{enumerate}
\usepackage{amsthm}
\usepackage{amsmath}
\usepackage{amssymb}
\usepackage{amsfonts}

\newtheorem{thm}{Theorem}[section]
\newtheorem*{thm*}{Theorem} 
\newtheorem{lem}[thm]{Lemma}
\newtheorem{Def}[thm]{Definition}

\newtheorem{prop}[thm]{Proposition}
\newtheorem{cor}[thm]{Corollary}
\newtheorem{rem}[thm]{Remark}

\theoremstyle{definition}{\newtheorem{ex}[thm]{Example}}

\newcommand{\Z}{\ensuremath{\mathbb{Z}}}

\newcommand{\s}{\ensuremath{\sigma}}

\newcommand{\te}{\ensuremath{\otimes_k}}

\newcommand{\G}{\ensuremath{\mathcal{G}}}
\newcommand{\Gd}{{\ensuremath{^{\sigma^d \! \!}\mathcal{G}}}}
\newcommand{\X}{\ensuremath{\mathcal{X}}}
\newcommand{\Y}{\ensuremath{\mathcal{Y}}}
\newcommand{\Xd}{{\ensuremath{^{\sigma^d \! \!}\mathcal{X}}}}
\newcommand{\GL}{\operatorname{GL}}
\newcommand{\SL}{\operatorname{SL}}
\newcommand{\Gm}{\mathbb{G}_m}
\newcommand{\Ga}{\mathbb{G}_a}

\newcommand{\Rd}{{\ensuremath{_{\sigma^d \! }R}}}
\newcommand{\Sd}{{\ensuremath{_{\s^d \!}S}}}

\newcommand{\Hom}{\operatorname{Hom}}
\newcommand{\Aut}{\operatorname{Aut}}

\newcommand{\Spec}{\operatorname{Spec}}
\newcommand{\Id}{\operatorname{id}}

\newcommand{\Alg}{\mathbf{Alg}}
\newcommand{\Sets}{\mathbf{Sets}}
\newcommand{\Groups}{\mathbf{Groups}}

\newcommand{\h}{\operatorname{H}_\s}
\newcommand{\z}{\operatorname{Z}_\s}
\newcommand{\del}{\partial}

\newcommand{\LL}{\mathcal{L}}
\newcommand{\ks}{$k$-$\s$}
\newcommand{\ds}{\delta\sigma}
\newcommand{\de}{\delta}
\newcommand{\ida}{\mathfrak{a}}
\newcommand{\I}{\mathbb{I}}
\newcommand{\id}{\operatorname{id}}

\title{Torsors for Difference Algebraic Groups}
\author{Annette Bachmayr}
\email{annette.bachmayr@tu-dortmund.de}
\address{(ne\'{e} Maier), Fakult\"{a}t f\"{u}r Mathematik, Technische Universit\"{a}t Dortmund, D-44221 Dortmund,	Germany}

\author{Michael Wibmer}
\email{wibmer@math.upenn.edu}
\address{Department of Marthematics, University of Pennsylvania, David Rittenhouse Laboratory, 209 South 33rd Street, Philadelphia, PA 19104, USA}

\thanks{The first author was funded by the Deutsche Forschungsgemeinschaft (DFG) - grant MA6868/1-1.}

\date{\today}

\begin{document}

	\begin{abstract}
		We introduce a cohomology set for groups defined by algebraic difference equations and show that it classifies torsors under the group action. This allows us to compute all torsors for large classes of groups. We also develop some tools for difference algebraic geometry and present an application to the Galois theory of differential equations depending on a discrete parameter.
	\end{abstract}
	
\maketitle

\section*{Introduction}
Groups defined by algebraic difference equations have been studied in connection with applications such as the Manin-Mumford conjecture (see e.g. \cite{Hrushovski:Manin-Mumford}, \cite{KowalskiPillay:algebraicsgroups}). More recently, Galois theories where the Galois groups are defined by algebraic difference equations have been developed (\cite{DHW}, \cite{OchinnikovWibmer:sGaloistheory}). In these Galois theories, the solution rings, when interpreted geometrically, are torsors under the action of the Galois group. It is therefore highly desirable to obtain a thorough understanding of the structure of torsors for groups defined by difference equations.

For algebraic groups, a starting point in the quest to understand the torsors under the action of a given algebraic group $G$ over a field $k$ is the canonical bijection between the (non-abelian) cohomology set
$\operatorname{H}^1(k,G)$ and the 
isomorphism classes of $G$-torsors. In this article we introduce a cohomology set
$\h^1(k,G)$ for a group $G$ defined by algebraic difference equations with coefficients in a difference field $k$, i.e., $k$ is a field equipped with an endomorphism $\s\colon k\to k$. 
We show that indeed $\h^1(k,G)$ classifies $G$-torsors. This allows us to explicitly describe all $G$-torsors for large classes of examples. 

As an application to the Galois theory of linear differential equations depending on a discrete parameter, we show that our cohomology set can also be used to classify the solutions rings (i.e., the $\s$-Picard-Vessiot rings) for a linear differential equation depending on a discrete parameter. Further applications to this Galois theory relating to the structure of the solution rings and the inverse problem will be presented in a future paper.

\medskip

Under the heading ``constrained cohomology''
E. Kolchin introduced in \cite{Kolchin} (see also \cite{Kovacic:constrainedcohomolgy}) a cohomology set for groups defined by algebraic differential equations that serves a similar purpose as our cohomology set for groups defined by algebraic difference equations. The approach taken by Kolchin is in spirit close to the Galois cohomology for algebraic groups and the term ``constrained'' refers to the fact that cocycles are certain maps from the automorphism group of the constrained closure (nowadays more commonly referred to as a ``differential closure'') of the base differential field into the group. 

For a given difference field $k$, it seems that so far no object that could play the role of a ``difference closure of $k$'' or of its automorphism group has been found. We completely avoid this issue by employing a functorial approach, somewhat in the spirit of how faithfully flat descent generalizes and replaces Galois descent. In particular, for us, cocylces for a group $G$ defined by algebraic difference equations with coefficients in a difference field $k$ are certain elements in $G(A\otimes_k A)$ where $A$ is a difference algebra over $k$.  

Among other things, this has the advantage that we are not restricted to certain nice extensions of difference fields (playing the role of Galois extensions in the classical theory) but we can consider arbitrary difference algebras $A$ instead.
In particular, the cohomology set $\h^1(A/k,G)$ classifies $G$-torsors that have a point in $A$.
Indeed, this works for arbitrary difference rings $k$, as long as $A$ is faithfully flat over $k$.

It seems clear that our approach could be adopted to groups defined by algebraic differential equations as well. This would eliminate the need to consider automorphism of the differential closure and would expedite the calculation of examples.

\medskip

The outline of the article is as follows: After explaining the basic definitions (such as difference algebraic groups and torsors), we introduce the method of faithfully flat descent in difference algebraic geometry. This is a crucial tool for the proof of our main result (Theorem \ref{Klassifikation}; the bijection between the cohomology set and the isomorphism classes of torsors). While for the purpose of this paper faithfully flat descent for difference varieties would be sufficient, we develop faithfully flat descent in the more general context of sheaves since we believe this to be important for the applicability of the method to other problems in difference algebraic geometry. 

In Section \ref{sec 3} we introduce the cohomology set $\h^1(A/k,G)$ and we show that it coincides with the usual cohomology of an algebraic group in the case when $G$ is defined by algebraic equations (rather than actual difference equations).

In Section \ref{sec: torsors and cocycles} we then prove the main result: The bijection between $\h^1(A/k,G)$ and the isomorphism classes of $G$-torsors that are trivial over $A$. In Section \ref{sec 5} we introduce the cohomology set $\h^1(k,G)$ and we show that a short exact sequence of groups defined by difference equations gives rise to a long exact sequence of cohomology sets. We also explicitly compute $\h^1(k,G)$ for some examples.

In Section \ref{C} we use the results of the previous sections to write down explicitly all torsors for a large class of groups.

Finally, in the last section we present an application to the Galois theory of linear differential equations depending on a discrete parameter. We show that if $G$ is the $\s$-Galois group of a linear differential equation $y'=Ay$ (with respect to some $\s$-Picard-Vessiot ring), then the isomorphism classes of all $\s$-Picard-Vessiot rings for $y'=Ay$ are classified by $\h^1(k,G)$. 
\begin{acknowledgements}
	We are thankful to Johannes Schmidt for helpful comments.
\end{acknowledgements}

\section{Notation and basic definitions} 
We first recall some basic notions from difference algebra. Standard references are \cite{Levin} and \cite{Cohn}. A look at the first chapters of \cite{Wibmer:Habil} might also be helpful.

All rings are assumed to be commutative and with unity. A difference ring, or \textit{$\s$-ring} for short, is a ring $k$ equipped with a ring homomorphism $\s\colon k \to k$. If the underlying ring of a $\s$-ring is a field, it is called a \emph{$\s$-field}. A morphism of $\s$-rings is a morphism of rings that commutes with the action of $\s$. If $k\to R$ is a morphism of $\s$-rings, we also say that $R$ is a \emph{\ks-algebra}. A morphism of \ks-algebras is a morphism of $k$-algebras that is also a morphism of difference rings. The category of \ks-algebras is denoted by \ks-$\Alg$. A \emph{\ks-subalgebra} of a \ks-algebra is a $k$-subalgebra that is stable under $\s$. If $R$ is a \ks-algebra, and $B\subseteq R$ a subset, the smallest \ks-subalgebra of $R$ containing $B$ is denoted by $k\{B\}$. A \ks-algebra $R$ is called \emph{finitely $\s$-generated} if $R=k\{B\}$ for a finite subset $B$ of $R$. 
If $R$ and $S$ are \ks-algebras, then $R\otimes_k S$ is naturally a \ks-algebra via $\s(r\otimes s)=\s(r)\otimes\s(s)$.
A \emph{$\s$-field extension} $L/K$ is an extension of fields such that $L$ is a $K$-$\s$-algebra.

An ideal $\ida$ in a $\s$-ring $R$ is called a \emph{$\s$-ideal} if $\s(\ida)\subseteq\ida$. In this case $R/\ida$ naturally carries the structure of a $\s$-ring.
The \textit{$\s$-polynomial ring} $k\{y_1,\dots,y_n\}$ over a $\s$-ring $k$ in the $\s$-variables $y_1,\dots,y_n$ is the polynomial ring over $k$ in the variables $y_1,\dots,y_n, \s(y_1),\dots,\s(y_n),\s^2(y_1),\dots$. We consider $k\{y_1,\dots,y_n\}$ as a \ks-algebra by extending $\s$ from $k$ to $k\{y_1,\dots,y_n\}$ as suggested by the names of the variables.
 
A set of $\s$-polynomials $F \subseteq k\{y_1,\dots,y_n\}$ defines a functor $X$ from $k$-$\s$-$\Alg$ to $\Sets$ by $X(R)=\{ x \in R^n \mid f(x)=0 \text{ for all } f \in F \}$ for all $k$-$\s$-algebras $R$. 
 Any such functor will be called a \emph{\ks-variety} (or a $\s$-variety over $k$). A morphism of \ks-varieties is a morphism of functors.
 
For a $\s$-variety $X$ we define 
$$\I(X)=\{f\in k\{y_1,\dots,y_n\}|\ f(x)=0 \text{ for all } x\in X(R) \text{ and \ks-algebras } R \}.$$
Then $\I(X)$ is a $\s$-ideal in $k\{y_1,\dots,y_n\}$ and 
 $$k\{X\}=k\{y_1,\dots,y_n\}/\I(X)$$
 is a finitely $\s$-generated \ks-algebra that we call the \emph{coordinate ring} of $X$. 
 For any \ks-algebra $R$ we have a map
 $\Hom(k\{X\},R)\to X(R)$
 given by sending a morphism of \ks-algebras $\psi\colon k\{X\}\to R$ to $(\psi(\overline{y_1}),\ldots,\psi(\overline{y_n}))$. It is easy to see (Cf. \cite[Section 1.2.1]{Wibmer:Habil}.) that this defines an isomorphism of functors $\Hom(k\{X\},-)\cong X$ and that the category of \ks-varieties is anti-equivalent to the category of finitely generated \ks-algebras. In particular, a morphism $\phi\colon X\to Y$ of difference varieties corresponds to a morphism $\phi^*\colon k\{Y\}\to k\{X\}$ of \ks-algebras. In the sequel we will sometimes identify $X(R)$ with $\Hom(k\{X\},R)$.
 
 We allow ourselves the freedom of calling any functor from \ks-$\Alg$ to $\Sets$ that is representable by a finitely $\s$-generated \ks-algebra a \ks-variety.

A \emph{$\s$-closed $\s$-subvariety} $Y$ of a $\s$-variety $X$, is a subfunctor $Y$ of $X$ defined by a $\s$-ideal. More precisely, we require that there exists a $\s$-ideal $\ida\subseteq k\{X\}$ such that
$$ 
\xymatrix{
Y(R) \ar@{^{(}->}[r] \ar_\cong[d] & X(R) \ar^\cong[d] \\
\Hom(k\{X\}/\ida, R) \ar@{^{(}->}[r] & \Hom(k\{X\},R)
}
$$
commutes for every \ks-algebra $R$. For example, if $F\subseteq F'\subseteq k\{y_1,\ldots,y_n\}$, then the $\s$-variety defined by $F'$ is a $\s$-closed $\s$-subvariety of the $\s$-variety defined by $F$.

A morphism $\phi\colon X\to Y$ of $\s$-varieties is called a \emph{$\s$-closed embedding} if it induces an isomorphism between $X$ and a $\s$-closed $\s$-subvariety of $Y$. One can show (\cite[Lemma 1.2.7]{Wibmer:Habil}) that $\phi\colon X\to Y$ is a $\s$-closed embedding if and only if $\phi^*\colon k\{Y\}\to k\{X\}$ is surjective.

If $R$ is a $\s$-ring or a $k$-$\s$-algebra, we write $R^\sharp$ for the ring or $k$-algebra obtained from $R$ by forgetting $\s$. If $\X$ is an affine scheme of finite type over $k$, then we can define a functor
$X$ from \ks-$\Alg$ to $\Sets$ by $X(R)=\X(R^\sharp)$ for any \ks-algebra $R$. It is easy to see that $X$ is a \ks-variety: If $\X=\Spec(k[y_1,\ldots,y_n]/\ida)$ for some ideal $\ida\subseteq k[y_1,\ldots,y_n]$, then $X$ is the $\s$-variety defined by $\ida\subseteq k\{y_1,\ldots,y_n\}$.

If $X$ is any functor from \ks-$\Alg$ to $\Sets$ and $A$ a \ks-algebra, we denote by $X_A$ the base change functor from $A$-$\s$-$\Alg$ to $\Sets$ given by $X_A(R)=X(R)$ where the $A$-$\s$-algebra $R$ is interpreted as a \ks-algebra via $k\to A\to R$. If $X$ is a \ks-variety $X_A$ is an $A$-$\s$-variety (represented by $k\{X\}\otimes_k A$).

 The category of \ks-varieties has products, indeed $k\{X\times Y\}=k\{X\}\otimes_k k\{Y\}$. A \textit{$\s$-algebraic group} over $k$ is a group object in the category of \ks-varieties.
 A morphism $\phi\colon G\to H$ of $\s$-algebraic groups is a morphism of $\s$-varieties that is compatible with the group structure, i.e., $\phi_R\colon G(R)\to H(R)$ is a morphism of groups for all \ks-algebras $R$.
A \emph{$\s$-closed subgroup} $H$ of a $\s$-algebraic group $G$, is a $\s$-closed $\s$-subvariety $H$ of $G$ such that $H(R)$ is a subgroup of $G(R)$ for any \ks-algebra $R$. If $H(R)$ is a normal subgroup of $G(R)$ for all \ks-algebras $R$, then $H$ is called a \emph{normal} $\s$-closed subgroup.

\begin{ex} \label{ex algebr Gruppe}
Let $k$ be a $\s$-ring and let $\G$ be an algebraic group over $k$, by which we mean an affine group scheme of finite type over $k^\sharp$. Then $\G$ can be interpreted as a $\s$-algebraic group over $k$ by considering the functor $G$ from $k$-$\s$-$\Alg$ to $\Groups$ with $G(R)=\G(R^\sharp)$ for every $k$-$\s$-algebra $R$.
\end{ex}

\begin{Def} \label{def: torsor} Let $k$ be a $\s$-ring and $G$ a $\s$-algebraic group over $k$.
A \emph{$G$-torsor} is a non-empty $\s$-variety $X$ over $k$ together with a morphism $\mu \colon G \times X \to X$ of $\s$-varieties such that $\mu_R\colon G(R)\times X(R) \to X(R)$ defines a group action of $G(R)$ on $X(R)$ and such that $ G(R)\times X(R) \to X(R) \times X(R), \ (g,x) \mapsto (\mu_R(g,x), x)$ is bijective for every $k$-$\s$-algebra $R$.
\end{Def} 

 A \emph{morphism of $G$-torsors} $\phi\colon X\to Y$ is a morphism of $\s$-varieties that is $G$-equivariant, i.e.,
 $$
 \xymatrix{
 	G\times X \ar[r] \ar_{\id\times \phi}[d]  & X \ar^\phi[d] \\
 	G\times Y \ar[r] & Y
 } 
 $$
 commutes.
 
If $X$ is a $G$-torsor and $R$ is a $k$-$\s$-algebra, we will abbreviate $g.x=\mu_R(g,x)$ for $g \in G(R)$, $x \in X(R)$. If $S$ is another $k$-$\s$-algebra and if there is an obvious morphism $R \to S$ (which will be clear from the context) we will use expressions such as $g.x$ for $g \in G(S), x \in X(R)$ to denote $g.x=g.\alpha(x)$ where $\alpha$ is the induced map $\alpha \colon X(R) \to X(S)$. 

The bijectivity of $G(R)\times X(R) \to X(R) \times X(R)$ in Definition \ref{def: torsor} is equivalent to saying that for $x_1,x_2\in X(R)$ there exists a unique $g\in G(R)$ with $x_2=g.x_1$.
The group itself $X=G$ is a $G$-torsor under the group multiplication (from the left).  

Let $A$ be a $k$-$\s$-algebra and let $X$ be a $G$-torsor over $k$. Then $X_A$ is a $G_A$-torsor and we say that \emph{$X$ is trivial over $A$} if $X_A$ is isomorphic to $G_A$ as a $G_A$-torsors.  

\begin{rem}\label{trivialer Torsor}
 A $G$-torsor $X$ is trivial over $A$ if and only if $X(A)\neq \emptyset$.
\end{rem}
\begin{proof}
 If $X$ is trivial over $A$, then $X(A)$ can be identified with $G(A)$ and is thus non-empty (since $1\in G(A)$). If $X(A)$ is non-empty, let $x \in X(A)$. For any $A$-$\s$-algebra $R$, set $\gamma_R \colon G(R) \to X(R), \ g \mapsto g.x $. This defines an isomorphism of $G_A$-torsors $\gamma \colon G_A \to X_A$.
\end{proof} 

Clearly every $G$-torsor $X$ becomes trivial over a finitely $\s$-generated $k$-$\s$-algebra $A$, namely $A=k\{X\}$. However, as the following example illustrates, if $k$ is a $\s$-field and $X$ a $G$-torsor, it is in general not true that $X$ becomes trivial over a $\s$-field extension of $k$. Thus, even though we are mostly interested in the case where $k$ is a $\s$-field, it is important to work with \ks-algebras $A$ rather than $\s$-field extensions of $k$.

\begin{ex} \label{ex1} Let $k$ be a $\s$-field and let $G$ be the $\s$-closed subgroup of the multiplicative group $\Gm$ given by
	$$G(R)=\{g\in R^\times|\ g^2=1,\ \sigma(g)=g\}$$
	for any \ks-algebra $R$.
{\begin{enumerate} 
 \item   Let $X$ be the \ks-variety given by $X(R)=\{x \in R \mid x^2=1,  \ \s(x)=-x \}$ for every $k$-$\s$-algebra $R$. Then $X$ is a $G$-torsor ($G(R)$ acts on $X(R)$ via multiplication in $R$) and $X$ is non-trivial over every $\s$-field extension $K/k$, since $X(K)=\emptyset$ for all such $K$. 
\item More generally, if $a, b \in k^\times$ such that $\s(a)=ab^2$, then $X_{a,b}(R)=\{ x\in R \mid x^2=a, \ \sigma(x)=bx\}$ defines a $G$-torsor $X_{a,b}$. (Note that the condition $\s(a)=ab^2$ guarantees that $X$ is not empty.) Based on our study of $H^1$ we will see later that in fact every $G$-torsor is isomorphic to $X_{a,b}$ for a suitable choice of $a, b \in k^\times$ (Example \ref{ex3}).
\end{enumerate}}
\end{ex}

\begin{ex} \label{ex: Ga intro}
	Let $k$ be a $\s$-field and let $G$ be the $\s$-closed subgroup of the additive group $\Ga$ given by
	$$G(R)=\{g\in R|\ \s^n(g)+\lambda_{n-1}\s^{n-1}(g)+\ldots+\lambda_0g=0\}$$
	for any \ks-algebra $R$, where $\lambda_0,\ldots,\lambda_{n-1}\in k$. For $a\in k$ let $X_a$ be given by
	$$X_a(R)=\{x\in R|\ \s^n(x)+\lambda_{n-1}\s^{n-1}(x)+\ldots+\lambda_0x=a \}$$
	for any \ks-algebra $R$. Then $X_a$ is a $G$-torsor under the action $g.x=g+x$. Based on our study of $H^1$ we will see later that in fact every $G$-torsor is isomorphic to some $X_{a}$ (Example \ref{ex Ga}). 
\end{ex}

\section{Faithfully flat descent for difference varieties}
In this section, we introduce faithfully flat descent for difference varieties. This is a crucial tool for proving the main result (Theorem \ref{Klassifikation}) in Section 
\ref{sec: torsors and cocycles}, but we also expect this technique to be applicable to other problems in difference algebraic geometry. Of course, this is inspired by faithfully flat descent for schemes. (See e.g., \cite[Tag 0238]{stacks-project}).
We first treat the more general case of sheaves in the context of faithfully flat difference algebras.

 Throughout this section, $k$ denotes a $\s$-ring and $A$ denotes a faithfully flat $k$-$\s$-algebra, i.e., $A$ is a \ks-algebra and $A$ is faithfully flat as a $k$-algebra.

Sheaves are useful when working with quotients by group actions. (See e.g. \cite[Chapter III]{DG} or \cite[Section 5.1]{Wibmer:Habil}.)
A functor $F$ from $k$-$\s$-$\Alg$ to $\Sets$ is called a \emph{sheaf} over $k$ if it satisfies the following two conditions:
\begin{enumerate}
	\item For every morphism of \ks-algebras $R\to S$ such that $S$ is a faithfully flat $R$-algebra, the sequence $$F(R) \to F(S) \rightrightarrows F(S\otimes_R S)$$ is exact (i.e., an equalizer of sets). In particular, $F(R) \to F(S)$ is injective. 
	\item For every finite family $(R_i)$ of \ks-algebras, the map $F(\prod R_i)\to\prod F(R_i)$ induced from the projections $\prod R_i\to R_{i_0}$ is bijective.
\end{enumerate}

It is easy to see that a representable functor is a sheaf (cf. \cite[Cor. 5.1.4]{Wibmer:Habil}). In particular any \ks-variety is a sheaf.  A morphism of sheaves is a morphism of functors.


If $F$ is a functor from $A$-$\s$-$\Alg$ to $\Sets$ and $B$ is a $k$-$\s$-algebra, then we write $B\te F$ for the base change $B\te F\colon (B\otimes_k A)$-$\s$-$\Alg\to \Sets$ of $F$ from $A$ to $B\otimes_k A$. 
That is, for a $B\otimes_k A$-$\s$-algebra $R$ we have
$(B\otimes_k F)(R)=F(R)$, where $R$ is considered as an $A$-$\s$-algebra via $A\to A\otimes_k B\to R$. Similarly we define $F\otimes_k B\colon (A\otimes_k B)$-$\s$-$\Alg\to \Sets$. If $F$ is a sheaf over $A$, then $B\te F$ is a sheaf over $B \te A$.

Let $F \colon A$-$\s$-$\Alg \to \Sets$ be a sheaf over $A$. A \textit{descent datum on $F$} (or more precisely, an $A/k$-descent datum) is an isomorphism $\phi \colon A \te F \to F \te A$ such that $\phi_{13}=\phi_{12}\circ \phi_{23}$, where 
\begin{eqnarray*}
 \phi_{23} \colon A\te A \te F \to A \te F \te A \\
 \phi_{12} \colon A \te F \te A \to F \te A \te A \\
 \phi_{13} \colon A\te A \te F \to F \te A \te A
\end{eqnarray*} denote the canonical extensions of $\phi$.
The isomorphisms $\phi_{ij}$ of functors from $A\otimes_k A\otimes_k A$-$\s$-$\Alg$ to $\Sets$ correspond to the three canonical possibilities of considering $A\otimes_k A\otimes_k A$ as an $A\otimes_k A$-algebra as determined by the indices $ij$. For example, $\phi_{23}$ corresponds to $A\otimes_k A\to A\otimes_k A\otimes_k A,\ a\otimes b\mapsto 1\otimes a\otimes b$.  

 Let $\tilde F \colon A$-$\s$-$\Alg \to \Sets$ be another sheaf with descent datum $\tilde \phi \colon A \te \tilde F \to \tilde F \te A$. A \textit{morphism of sheaves with descent data} is a morphism of functors $\gamma \colon F \to \tilde F$ such that the following diagram commutes: 
\[\xymatrix{A\te F \ar@{->}[d]_{A\otimes\gamma} \ar@{->}[rr]^{\ \phi \ } &&F\te A \ar@{->}[d]^{\gamma \otimes A}\\
            A \te \tilde F \ar@{->}[rr]_{\tilde \phi} &&  \tilde F\te A} \] 
        In this situation we also say that $\gamma$ is compatible with the descent data.
If $F_0  \colon k$-$\s$-$\Alg \to \Sets$ is a sheaf over $k$, then there is a canonical $A/k$-descent datum on $(F_0)_A$ corresponding to the fact that the two maps $A\rightrightarrows A\otimes_k A$ agree on $k$.

\begin{thm}\label{abstieg}
 The base change from $k$ to $A$ defines an equivalence of categories from the category of sheaves over $k$ to the category of sheaves over $A$ with $A/k$-descent data.
\end{thm}
\begin{proof}
 Let $F$ be a sheaf over $A$ with descent datum $\phi \colon A \te F \to F \te A$. We assign to $(F,\phi)$ a functor $F_0 \colon k$-$\s$-$\Alg\to\Sets$ as follows. For every $k$-$\s$-algebra $R$, consider the diagram 
\begin{equation}
\label{eqn: diagram descent}
\xymatrix{ (A\te F)(R \te A \te A)  \ar@{->}[rr]^{\ \phi_{R\otimes  A \otimes A} \ } && (F\te A)(R \te A \te A)   \\
            & F(R\te A)  \ar@{->}[ul]^{\psi_1} \ar@{->}[ur]_{\psi_2} & } \end{equation}
    where $\psi_1$ and $\psi_2$ are obtained from the morphisms $R \te A\rightarrow R \te A\te A$ that map $r\otimes a$ to $r\otimes 1 \otimes a$, or to $r\otimes a\otimes 1$; and define 
\[F_0(R)=\{x \in F(R\te A) \mid \phi_{R\otimes A\otimes A}(\psi_1(x))=\psi_2(x)\}. \]
We claim that $F_0$ is a sheaf over $k$ such that $(F_0)_A$ is isomorphic to $F$ as a sheaf over $A$ with descent datum. Let $R \to S$ be a faithfully flat morphism of $k$-$\s$-algebras. Then $R\otimes_k A \to S\otimes_k A$ is a faithfully flat morphisms of $A$-$\s$-algebras, so 
\[F(R\te A) \to F(S\te A) \rightrightarrows F((S\te A)\otimes_{(R\otimes_k A)}(S\te A)) \] is exact, since $F$ is a sheaf. Thus the sequence 
\[F(R\te A) \to F(S\te A) \rightrightarrows F(S\otimes_R S\te A) \] is exact and it follows that 
\[F_0(R) \to F_0(S) \rightrightarrows F_0(S\otimes_R S) \] is exact.
To verify the second property in the definition of a sheaf, let $(R_i)$ be a finite family of \ks-algebras. Using the compatibility of tensor products and products and the fact that $F$ is a sheaf, diagram (\ref{eqn: diagram descent}) (with $\prod R_i$ in place of $R$) becomes

$$
\xymatrix{ \prod(A\te F)(R_i \te A \te A)  \ar@{->}[rr] && \prod(F\te A)(R_i \te A \te A)   \\
	& \prod F(R_i\te A)  \ar@{->}[ul] \ar@{->}[ur] & } $$
This shows that $F_0$ is a sheaf.

To prove that $(F_0)_A$ is isomorphic to $F$ as a sheaf with descent datum, we show that for every $A$-$\s$-algebra $S$, the bijection $\phi_{S\otimes A} \colon F(S\te A) \to F(S\te A)$ restricts to a bijection $F_0(S) \to F(S)$. (Note that $F(S)$ injects into $F(S\te A)$ since $S\to S\te A$ is faithfully flat.) The two canonical morphisms $S\te A \rightrightarrows S\te A \te A$ induce maps 
\begin{eqnarray*}
\beta_1\colon (A\te F)(S\te A) \to (A\te A\te F)(S\te A \te A) \\
\beta_2\colon (A\te F)(S \te A) \to (A\te F\te A)(S\te A\te A) \\
\gamma_1\colon (F\te A)(S \te A) \to (F\te A\te A)(S\te A \te A) \\
\gamma_2\colon (F\te A)(S \te A) \to (F\te A\te A)(S\te A \te A) 
\end{eqnarray*}
and we consider the following diagram.
\[\xymatrix{ (A\te F)(S \te A ) \ar@{->}[d]_{\ \phi_{S\otimes A} \ } \ar@<2pt>[rr]^{\beta_1 \ \ \ \ } \ar@<-2pt>[rr]_{\ \phi_{23}^{-1}\circ\beta_2 \ \ \ \ \ \ \  }  && (A\te A\te  F)(S \te A \te A) \ar@{->}[d]^{\ \phi_{13} \ } \\
             (F\te A)(S \te A )  \ar@<2pt>[rr]^{\ \gamma_1 \ \ \ \ \ \ \  } \ar@<-2pt>[rr]_{\ \gamma_2 \ \ \ \ \ \ \  } && (F\te A\te  A)(S \te A \te A) } \] The diagram commutes with the top arrows and the cocycle condition implies that it also commutes with the bottom arrows. As the vertical arrows are isomorphisms, we conclude that $\phi_{S\otimes A}$ maps the equalizer of the upper horizontal arrows to the equalizer of the lower horizontal arrows. The upper equalizer equals $F_0(S)$ and the lower equalizer can be identified with $F(S)$, since $F$ is a sheaf and $S\te A$ is faithfully flat over $S$. Hence $(F_0)_A$ is isomorphic to $F$ and the cocycle condition implies that this isomorphism is compatible with the descent data.

We conclude that the base change functor is essentially surjective. On the other hand, if $F_0$ and $G_0$ are sheaves over $k$ and if we denote $F=(F_0)_A$ and $G=(G_0)_A$, then it is easy to check that every morphism $F\to G$ of sheaves with $A/k$-descent data is the base change of a unique morphism $F_0\to G_0$ of sheaves over $k$. Hence the base change functor is full and faithful and the claim follows. 
\end{proof}

\begin{lem}\label{Abstieg von affinen Schemata}
 Let $F_0\colon k$-$\s$-$\Alg \to \Sets$ be a sheaf over $k$ such that $F=(F_0)_A$ is representable, that is, there exists an $A$-$\s$-algebra $B$ such that $F\cong\Hom(B,-)$. Then $F_0$ is representable.
\end{lem}

\begin{proof}
 Let $\phi \colon A\te F \to F \te A$ denote the canonical descent datum on $F=(F_0)_A$ and let $\phi^* \colon B\te A \to A \te B$ be the corresponding isomorphism of $(A\te A)$-$\s$-algebras. We consider the $k$-$\s$-algebra \[B_0=\{b \in B \mid \phi^*(b\otimes 1)=1\otimes b\}.\] We claim that $\Hom(B_0,-)\cong F_0$. Theorem \ref{abstieg} implies that it suffices to show that $\Hom(A\te B_0,-)\cong \Hom(B,-)$ as sheaves over $A$ with $A/k$-descent data.
 Using the $(A\te A\te A)$-$\s$-isomorphisms
 \begin{eqnarray*}
 	\phi_{23}^*=(\phi_{23})^*\colon A\te B \te A \to A \te A \te B \\
 	\phi_{12}^*=(\phi_{12})^*\colon B\te A \te A \to A \te B \te A \\
 	\phi_{13}^*=(\phi_{13})^*\colon B\te A \te A \to A \te A \te B 
 \end{eqnarray*}
 the latter statement can be rephrased using morphisms of $\s$-rings only.
 
  We have to show that the canonical map $A\otimes_k B_0\to B$ is an isomorphism compatible with the descent data. If we disregard the action of $\s$ this statement is well known. (See e.g., \cite[Tag 0238]{stacks-project}.) So the claim follows since all the involved maps are compatible with $\s$. 
 
\end{proof}

\begin{cor}\label{cor: Abstieg von Varietaeten}
	 The base change from $k$ to $A$ defines an equivalence of categories from the category of $\s$-varieties over $k$ to the category of $\s$-varieties over $A$ with $A/k$-descent data.
	
\end{cor}
\begin{proof}
	As the \ks-varieties are exactly the sheaves over $k$ that can be represented by a finitely $\s$-generated \ks-algebra, by Theorem \ref{abstieg} and Lemma \ref{Abstieg von affinen Schemata} it suffices to show the following: If $B_0$ is a \ks-algebra such that $B=B_0\otimes_k A$ is finitely $\s$-generated over $A$, then $B_0$ is finitely $\s$-generated over $k$.
	
As $B$ is $\s$-generated over $A$ by some finite subset $T$ of $B_0\otimes_k A$, we can modify $T$ if necessary, to assume $T\subseteq  B_0$.
Then $k\{T\}\subseteq B_0$ and $k\{T\}\te A =B_0\te A$. Hence $k\{T\}=B_0$, since $A$ is faithfully flat over $k$. 
\end{proof}

If $X$ and $Y$ are $A$-$\s$-varieties with $A/k$-descent data, $\phi\colon A\otimes_k X\to X\otimes_k A$ and $\rho \colon A\te Y \to Y \te A$ then $\phi\times \rho$ defines a descent datum on $X\times Y$ and this is the product in the category of $A$-$\s$-varieties with $A/k$-descent data.

\begin{rem} \label{rem: descent for torsors}
	By Corollary \ref{cor: Abstieg von Varietaeten} a $\s$-algebraic group $G_0$ over $k$ corresponds to a group object in the category of $\s$-varieties over $A$ with descent datum, i.e., to a $\s$-algebraic group $G$ over $A$ together with an $A/k$-descent datum such that the multiplication $G\times G\to G$, the inversion $G\to G$ and the identity $1\to G$ are compatible with the descent data. Similarly, if $X$ is a $G$-torsor equipped with a descent datum such that the action $G\times X\to X$ is compatible with the descent data, then $X$ descends to a $G_0$-torsor $X_0$, since the isomorphism $G\times X\to X\times X,\ (g,x)\mapsto (g.x, x)$ descends to an isomorphism  $G_0\times X_0\to X_0\times X_0$.
\end{rem}

For later use we record:

\begin{lem} \label{lem: sclosed embedding descends}
	Let $\phi\colon X\to Y$ be a morphism of \ks-varieties. If the base extension $\phi_A\colon X_A\to Y_A$ is a $\s$-closed embedding, then $\phi$ is a $\s$-closed embedding. 
\end{lem}
\begin{proof}
	If $k\{Y\}\otimes_k A\to k\{X\}\otimes_k A$ is surjective, then $k\{Y\}\to k\{X\}$ is surjective since $A$ is faithfully flat over $k$.
\end{proof}

\section{Cohomology}\label{sec 3}
In this section, we introduce the notion of cohomology sets for $\s$-algebraic groups. The definition is analogous to the cohomology sets for (not necessarily reduced) algebraic groups. (See e.g., \cite[III, \S4, 6.7]{DG} or \cite[Chapter III, \S 4]{Milne}.) Throughout this section, let $k$ be a $\s$-ring and let $G$ be a $\s$-algebraic group over $k$.\\
\\
Let $A$ be a $k$-$\s$-algebra. Consider the group homomorphisms
\[\delta_1, \delta_2 \colon G(A) \to G(A\te A),  \] where $\delta_1$ is induced by $A \to A\te A, \ a \mapsto 1\otimes a$ and $\delta_2$ is induced by $A \to A\te A, \ a \mapsto a\otimes 1$.
Similarly, consider the three group homomorphisms
\[\del_1, \del_2, \del_3 \colon G(A\te A) \to G(A\te A \te A), \] where $\del_1$ is induced by $A\te A \to A \te A \te A, \ a\otimes b \mapsto 1 \otimes a \otimes b$ etc.
\begin{Def}
Let $A$ be a $k$-$\s$-algebra.
The set of \emph{1-cocycles} of $G$ with respect to $A/k$ is defined as
\[\operatorname{Z}_\s^1(A/k,G)=\{\chi \in G(A\te A) \mid \del_2(\chi)=\del_1(\chi)\del_3(\chi) \}.\]
If $\chi$ is a cocycle and $\alpha\in G(A)$ it is straight forward to check that $\delta_1(\alpha)\chi\delta_2(\alpha)^{-1}$ is a cocycle.

 Two cocycles $\chi_1$ and $\chi_2$ are called \emph{equivalent}, if there is an $\alpha \in G(A)$ such that $\chi_1=\delta_1(\alpha)\chi_2\delta_2(\alpha)^{-1}$. The \emph{1-cohomology set} of $G$ with respect to $A/k$ is defined as the set of equivalence classes of 1-cocycles: \[\h^1(A/k, G)=\operatorname{Z}_\s^1(A/k,G)/\sim.\] 
\end{Def}


Recall that a pointed set is simply a set with a distinguished element. A morphism of pointed sets is a map of sets that preserves the distinguished element.

We consider $\h^1(A/k, G)$ as a pointed set, the distinguished element being the equivalence class of $1\in G(A\otimes_k A)$. If $G$ is abelian, the product of two cocycles is again a cocycle and the multiplication in $G(A\otimes_k A)$ induces a group structure on $\h^1(A/k, G)$ and we consider $\h^1(A/k, G)$ as an abelian group.


\begin{ex} \label{ex: H1 von Ga trivial}
If $A/k$ is faithfully flat, then $\h^1(A/k, \Ga)= \{0\}$. 
\end{ex}
\begin{proof}
	Here $\Ga$ is the $\s$-algebraic group given by $\Ga(R)=(R, +)$ for any \ks-algebra $R$. 
 The assertion follows directly from the fact that the following sequence is exact
\[\xymatrix{0 \ar@{->}[r] & k \ar@{->}[r] & A \ar@{->}[rr]^{ \delta_2-\delta_1 \ \ }  && A \te A   \ar@{->}[rr]^{\ \del_3-\del_2+\del_1 \ \ \ \ }   && A \te A \te A \ar@{->}[r]& \dots  } \]
 (See for example \cite[Chapter I, \S1, Section 2.7]{DG}.)
\end{proof}

\begin{rem} If $A \to B$ is a morphism of $k$-$\s$-algebras, then there is a natural map of pointed sets $\h^1(A/k, G)\to \h^1(B/k, G)$. We will see later (Lemma \ref{lem: compatibility}) that this map is injective and does not depend on the choice of the morphism $A \to B$.
%
 \end{rem}

\begin{rem} If $G \to H$ is a morphism of $\s$-algebraic groups, then there is a natural map of pointed sets $\h^1(A/k, G)\to \h^1(A/k, H)$.
 \end{rem}

\begin{prop}\label{algebr Gruppen}
Let $\G$ be an algebraic group over $k^\sharp$, that is, an affine group scheme of finite type over $k$ and let $G$ be the corresponding $\s$-algebraic group over $k$ (see Example \ref{ex algebr Gruppe}). Then 
\[\h^1(A/k,G)=\operatorname{H}^1(A^\sharp/k^\sharp, \G) \] for every $k$-$\s$-algebra $A$.
\end{prop}
\begin{proof}
Let $A$ be a $k$-$\s$ algebra. Then $\delta_1,\delta_2 $ induce morphisms $\delta_1^\sharp,\delta_2^\sharp \colon \G(A^\sharp) \to \G(A^\sharp \otimes_{k^\sharp} A^\sharp)$ and $\del_1,\del_2,\del_3$ induce morphisms $\del_1^\sharp,\del_2^\sharp,\del_3^\sharp \colon \G(A^\sharp\otimes_{k^\sharp} A^\sharp) \to \G(A^\sharp \otimes_{k^\sharp} A^\sharp \otimes_{k^\sharp} A^\sharp)$. 

The $1$-cohomology set $\operatorname{H}^1(A^\sharp/k^\sharp, \G)$ is defined as $\operatorname{H}^1(A^\sharp/k^\sharp, \G)=\operatorname{Z}^1(A^\sharp/k^\sharp, \G)/\sim$,
where $\operatorname{Z}^1(A^\sharp/k^\sharp, \G)=\{\chi \in \G(A^\sharp\otimes_{k^\sharp} A^\sharp) \mid \del_2^\sharp(\chi)=\del_1^\sharp(\chi)\del_3^\sharp(\chi) \}$ and two cocycles $\chi$, $\tilde \chi$ are called equivalent, if there is an $\alpha \in \G(A^\sharp)$ such that $\tilde \chi=\delta_1^\sharp(\alpha)\chi\delta_2^\sharp(\alpha)^{-1}$ (see for example \cite[Chapter III, \S4, Section 6.4]{DG}). Therefore, $\operatorname{Z}^1(A^\sharp/k^\sharp, \G)=\operatorname{Z}^1_\s(A/k, G)$ and $\operatorname{H}^1(A^\sharp/k^\sharp, \G)=\operatorname{H}^1_\s(A/k, G)$.
\end{proof}

\begin{cor} \label{cor: alg Gr triviale H1}
Let $k$ be a $\s$-field and let $A$ be a $k$-$\s$-algebra. Then $\h^1(A/k,G)$ is trivial for $G$ defined as in Proposition \ref{algebr Gruppen} for each of the following algebraic groups $\G$: $\Gm, \Ga,\GL_n, \SL_n$.
\end{cor}
\begin{proof}
Proposition \ref{algebr Gruppen} implies that it suffices to show that $\operatorname{H}^1(A^\sharp/k^\sharp,\G)$ is trivial.

Let $\Gamma$ denote the absolute Galois group of $k^\sharp$. It is well known that the Galois cohomology sets $\operatorname{H}^1(\Gamma,\G)$ are trivial for all of the above groups $\G$. (See e.g 
\cite[II.1.2, III.1.1]{serre} or \cite[Cor. 27.8 and 27.9]{milneiAG}.)
For smooth algebraic groups (and all our examples are smooth) these Galois cohomology sets classify $\G$-torsors (\cite[Chapter III, \S5, Section 3.6]{DG}).
As $\operatorname{H}^1(A^\sharp/k^\sharp,\G)$ classifies $\G$-torsors that become trivial over $A^\sharp$, also $\operatorname{H}^1(A^\sharp/k^\sharp,\G)$ has to be trivial.
\end{proof}

\begin{ex}\label{ex2}
Let $k$ be a $\s$-field and as in Example \ref{ex1} let $G$ be the $\s$-algebraic subgroup of the multiplicative group $\Gm$ given by
	 $$G(R)=\{g\in R^\times|\ g^2=1,\ \sigma(g)=g\}$$
	 for any \ks-algebra $R$.
We claim that for every $k$-$\s$-algebra $A$,
\[ \h^1(A/k, G)\cong\{\alpha \in A^\times \ | \ \alpha^2 \in k, \ \sigma(\alpha)/\alpha \in k \}/\sim  \]  where we write $\alpha \sim \alpha'$ if and only if there is a $\lambda \in k^\times$ with $(\alpha')^2=\lambda^2 \alpha^2$ and $\sigma(\alpha')/\alpha'=\sigma(\lambda)/\lambda \cdot \sigma(\alpha)/\alpha$.
Note that the right hand side is a group under multiplication and we claim to have an isomorphism of abelian groups.

If $\alpha\in A^\times$ satisfies $\alpha^2\in k$ and $\s(\alpha)/\alpha\in k$, then the element $\chi=\alpha^{-1}\otimes \alpha$ is contained in $G(A\te A)$ and it is a cocycle for $G$. Moreover, $\chi$ is equivalent to the trivial cocycle if and only if there exists a $\lambda \in k^\times$ with $\lambda\cdot \alpha\in G(A)$ or equivalently, if and only if $\alpha \sim 1$. The assignment $\alpha \mapsto \chi$ thus defines an injective group morphism from the right hand side to $\h^1(A/k, G)$.

Now let $\chi \in \operatorname{Z}^1_\s(A/k,G)\subseteq (A\te A)^\times$. Then $\chi$ can also be considered as an element in $\operatorname{Z}^1_\s(A/k,\Gm)$ and thus there is an $\alpha \in A^\times$ with $\alpha^{-1}\otimes \alpha=\chi$ by Corollary \ref{cor: alg Gr triviale H1}. Moreover, $\chi^2=1$ and $\sigma(\chi)=\chi$ imply
$\alpha^2 \in k$ and $\sigma(\alpha)/\alpha \in k$ and we obtain the desired bijection. 
\end{ex}

\begin{ex} \label{ex Ga H1(A)}
	Let $k$ be a $\s$-field and $\LL(y)=\s^n(y)+\lambda_{n-1}\s^{n-1}(y)+\ldots+\lambda_0y$ a linear difference equation over $k$. As in Example \ref{ex: Ga intro}, let $G$ be the difference algebraic subgroup of $\Ga$ defined by $\LL$, i.e.,
	$$G(R)=\{g\in R|\ \LL(g)=0\}\subseteq \Ga(R)$$
	for any \ks-algebra $R$. Let $A$ be a \ks-algebra.
	We claim that
	$$\h^1(A/k,G)\cong \{\alpha\in A|\ \LL(\alpha)\in k\}/\sim,$$
	as abelian groups, where $\alpha\sim \alpha'$ if and only if $\LL(\alpha'-\alpha)\in\LL(k)$.	
	
	If $\alpha\in A$ satisfies $\LL(\alpha)\in k$, then $\chi=1\otimes\alpha-\alpha\otimes 1\in G(A\otimes_k A)$ is a cocycle for $G$. Moreover, $\chi$ is equivalent to the trivial cocycle if and only if there exists a $\lambda \in k$ with $\alpha-\lambda\in G(A)$ and this is equivalent to $\alpha \sim 0$. The assignment $\alpha \mapsto \chi$ thus defines an injective group morphism from the right hand side to $\h^1(A/k, G)$.
	
Now let $\chi\in\operatorname{Z}_\s^1(A/k,G)$ be a cocycle. Since $\chi$ is trivial as a cocylce for $\Ga$ (Example \ref{ex: H1 von Ga trivial}), $\chi$ is of the form $\chi=1\otimes \alpha-\alpha\otimes 1$ for some $\alpha\in\Ga(A)=A$ and the claim follows.
\end{ex}

\begin{prop}\label{direktes Produkt}
Let $G$ and $H$ be $\s$-algebraic groups over $k$. Then there is a canonical bijective map $\h^1(A/k,G\times H) \to \h^1(A/k,G) \times \h^1(A/k, H)$.
\end{prop}
\begin{proof}
 The projections from $G\times H$ to $G$ and $H$ restrict to maps from $\operatorname{Z}_\s^1(A/k, G\times H)$ to $\operatorname{Z}_\s^1(A/k, G)$ and $\operatorname{Z}_\s^1(A/k, H)$ which induce maps from $\operatorname{H}_\s^1(A/k, G\times H)$ to $\operatorname{H}_\s^1(A/k, G)$ and $\operatorname{H}_\s^1(A/k, H)$. The resulting map  $\h^1(A/k,G\times H) \to \h^1(A/k,G) \times \h^1(A/k, H)$ is bijective. 
\end{proof}

\begin{ex} \label{ex:  Gan and Gmn trivial}
Combining Corollary \ref{cor: alg Gr triviale H1} with Proposition \ref{direktes Produkt}, we obtain: If $k$ is a $\s$-field then $\h^1(A/k,\Ga^n)$ and $\h^1(A/k,\Gm^n)$ are trivial for every $k$-$\s$-algebra $A$.
\end{ex}

 \section{Torsors and cocycles} \label{sec: torsors and cocycles}
In this section, we prove our main result which asserts that torsors of $\s$-algebraic groups can be classified by the corresponding cohomology sets. Throughout this section, let $k$ be a $\s$-ring, $G$ a $\s$-algebraic group over $k$ and $A$ a faithfully flat \ks-algebra.

Let $\Sigma_{A/k}(G)$ be the set of all $k$-isomorphism classes of $G$-torsors that are trivial over $A$. We consider $\Sigma_{A/k}(G)$ as a pointed set, the distinguished element being the class of the trivial $G$-torsors.

\begin{thm} \label{Klassifikation} 
 There is an isomorphism of pointed sets between between $\Sigma_{A/k}(G)$ and $\h^1(A/k, G)$. 
\end{thm}
\begin{proof}
\textit{First step: Construct a map $\Xi \colon \Sigma_{A/k}(G) \to \h^1(A/k,G)$. }

Let $X$ be a $G$-torsor that is trivial over $A$. Then $X$ defines a 1-cocycle as follows. Fix an element $x \in X(A)$ (see Remark \ref{trivialer Torsor}). Consider the two natural maps $f_1,f_2 \colon X(A) \to X(A\te A)$. More precisely, $f_1$ is induced by $A \to A\te A, \  a \mapsto 1 \otimes a$ and $f_2$ is induced by $A \to A\te A, \ a \mapsto a \otimes 1$. Let $\chi \in G(A\te A)$ be the unique element with $f_1(x)=\chi.f_2(x)$. We claim that $\del_2(\chi)=\del_1(\chi)\del_3(\chi)$. For the sake of convenience, we denote the three maps $X(A \te A) \to X(A\te A \te A)$ also by $\del_1,\del_2,\del_3$ and similarly the three maps $(G\times X)(A \te A) \to (G\times X)(A\te A \te A)$. Note that $\del_3\circ f_2=\del_2\circ f_2$, $\del_2\circ f_1=\del_1\circ f_1$, and $\del_3\circ f_1=\del_1\circ f_2$ on $X(A)$. We compute
\begin{eqnarray*}
 \del_2(\chi).\del_2(f_2(x))&=& \del_2(\chi.f_2(x))=\del_2(f_1(x))=\del_1(f_1(x))=\del_1(\chi.f_2(x)) \\
&=& \del_1(\chi).\del_1(f_2(x))=\del_1(\chi).\del_3(f_1(x))=\del_1(\chi).\del_3(\chi.f_2(x))\\
& =&\del_1(\chi).(\del_3(\chi).\del_3(f_2(x))) = (\del_1(\chi)\del_3(\chi)).\del_3(f_2(x)) \\
&=&(\del_1(\chi)\del_3(\chi)).\del_2(f_2(x))
\end{eqnarray*} and the claim follows. 

Now we show that another choice of a rational point $y\in X(A)$ yields a cocycle equivalent to $\chi$. Let $\alpha \in G(A)$ be the unique element with $y=\alpha.x$. Let $\tilde \chi \in G(A\te A)$ be the unique element with $f_1(y)=\tilde \chi.f_2(y)$. Then 
\begin{eqnarray*}
 (\delta_1(\alpha)\chi).f_2(x)&=& \delta_1(\alpha).f_1(x)=f_1(\alpha.x)=f_1(y)=\tilde \chi.f_2(y)=\tilde \chi.f_2(\alpha.x) \\
&=& \tilde \chi.(\delta_2(\alpha).f_2(x))=(\tilde \chi\delta_2(\alpha)).f_2(x)
\end{eqnarray*} and thus $\delta_1(\alpha)\chi=\tilde \chi\delta_2(\alpha)$, so $\chi \sim \tilde \chi$.

Let $Y$ be another $G$-torsor that is trivial over $A$ and such that there is an isomorphism $\gamma \colon X \to Y$ of $G$-torsors over $k$. Let $g_1,g_2 \colon Y(A) \to Y(A\te A)$ be the analogs of $f_1,f_2$. Then $\chi.f_2(x)=f_1(x)$ implies $\chi.\gamma_{A\otimes A}(f_2(x))=\gamma_{A\otimes A}(f_1(x))$, since $\gamma$ is $G$-equivariant. Using $\gamma_{A\otimes A}(f_i(x))=g_i(\gamma_A(x))$, we conclude that the rational point $\gamma_A(x) \in Y(A)$ defines precisely the same cocycle $\chi$ as $x \in X(A)$. 

Hence $X \mapsto \chi$ defines a map $\Xi \colon \Sigma_{A/k}(G) \to \h^1(A/k,G)$. 

\bigskip

\noindent \textit{Second step: Construct a map $\Upsilon \colon  \h^1(A/k,G) \to \Sigma_{A/k}(G)$. }\\
Let $\chi \in G(A\te A)$ be a $1$-cocycle of $G$ with respect to $A/k$. Define $Y=G_A$ and $\phi \colon A \te Y \to Y \te A$ by 
\[G(S) \to G(S), \ g \mapsto  g \cdot \chi \] for every $(A\te A)$-$\s$-algebra $S$. Then $\phi_{23}$ is given by multiplication with $\del_1(\chi)$ from the right, $\phi_{12}$ is given by multiplication with $\del_3(\chi)$ from the right and $\phi_{13}$ is given by multiplication with $\del_2(\chi)$ from the right. Thus the cocycle condition $\del_2(\chi)=\del_1(\chi)\del_3(\chi)$ implies $\phi_{13}=\phi_{12}\circ \phi_{23}$, and so $\phi$ is an $A/k$-descent datum on $Y$. We consider $Y$ as a $G_A$-torsor via $\mu \colon G_A\times Y \to Y$ given by 
$G(R)\times G(R) \to G(R), \ (g,y) \mapsto gy$
for every $A$-$\s$-algebra $R$. Let $\rho \colon A \te G_A \to G_A \te A$ be the canonical descent datum on $G_A$, so that $(G_A,\rho)$ descends to $G$, i.e., $\rho(g)=g$ for every $(A\te A)$-$\s$-algebra $R$ and every $g \in G(R)$. Then the following diagram commutes
\[\xymatrix{A\te (G_A\times Y) \ar@{->}[d]_{A\otimes \mu} \ar@{->}[rr]^{\ \rho \times \phi \ } &&(G_A\times Y)\te A \ar@{->}[d]^{\mu \otimes A}\\
            A \te Y \ar@{->}[rr]_{\phi} &&  Y\te A} \]
and thus $\mu$ is compatible with the descent data. Hence $Y$ descends to a $G$-torsor $X$ (Remark \ref{rem: descent for torsors}). Clearly, $X_A$ is isomorphic to $G_A$ as a $G_A$-torsor and so $X$ is trivial over $A$. 

Let $\tilde \chi \in \operatorname{Z}^1_\s(A/k,G)$ be a cocycle that is equivalent to $\chi$, i.e., there exists an $\alpha \in G(A)$ such that $ \chi = \delta_1(\alpha) \tilde \chi \delta_2(\alpha)^{-1}$. Define $\tilde Y=G_A$ with descent datum $\tilde \phi \colon A \te \tilde Y \to \tilde Y \te A$ given by multiplication with $\tilde \chi$ from the right. Then $\tilde Y$ descends to a $G$-torsor $\tilde X$ over $k$ and we claim that $X$ and $\tilde X$ are isomorphic as $G$-torsors. By Corollary \ref{cor: Abstieg von Varietaeten} it suffices to construct an isomorphism $\gamma \colon Y \to \tilde Y$ of $G_A$-torsors that is compatible with the descent data $\phi, \tilde \phi$. Define $\gamma \colon Y \to \tilde Y$ by 
\[ G_A(R) \to G_A(R), \ g \mapsto g \cdot \alpha \] for every $A$-$\s$-algebra $R$. Then $\gamma$ is an isomorphism of $G_A$-torsors. Note that $A\otimes \gamma \colon A \te Y \to A \te \tilde Y$ is given by multiplication with $\delta_1(\alpha)$ from the right and $\gamma\otimes A \colon Y \te A \to \tilde Y \te A$ is given by multiplication with $\delta_2(\alpha)$ from the right. The condition $ \chi = \delta_1(\alpha) \tilde \chi \delta_2(\alpha)^{-1}$ thus implies that the diagram 
\[\xymatrix{A\te Y \ar@{->}[d]_{A\otimes\gamma} \ar@{->}[rr]^{\ \phi \ } &&Y\te A \ar@{->}[d]^{\gamma \otimes A}\\
            A \te \tilde Y \ar@{->}[rr]_{\tilde \phi} &&  \tilde Y\te A} \] commutes, so $\gamma$ is compatible with the descent data $\phi, \tilde \phi$. 

Hence $\chi \mapsto X$ defines a map $\Upsilon \colon  \h^1(A/k,G) \to \Sigma_{A/k}(G)$. 

\bigskip

\noindent \textit{Third step: Show that $\Upsilon$ is the inverse of $\Xi$. }\\
Let $\tilde{X}$ be a $G$-torsor over $k$ that is trivial over $A$. Fix an arbitrary element $\tilde{x} \in \tilde{X}(A)$. Let $\chi \in G(A\te A)$ be the unique element with $f_1(\tilde{x})=\chi. f_2(\tilde{x})$, where $f_1$, $f_2$ are defined as in the first step. Now define $Y=G_A$ with $A/k$-descent datum $\phi \colon A \te Y \to Y \te A$ given by multiplication with $\chi$ from the right. We consider $Y$ as a $G_A$-torsor via multiplication from the left. By the previous step, $Y$ descends to a $G$-torsor $X$ over $k$. We claim that $X$ and $\tilde{X}$ are isomorphic as $G$-torsors.

 Let $\tilde \phi$ denote the canonical descent datum on $\tilde{X}_A$. By Corollary \ref{cor: Abstieg von Varietaeten} it suffices to construct an isomorphism $\gamma \colon Y \to \tilde{X}_A$ of $G_A$-torsors that is compatible with the descent data $\phi, \tilde \phi$. Define $\gamma \colon Y \to \tilde{X}_A$ by 
$G(R) \to\tilde{X}(R), \ g \mapsto g.\tilde{x}$ 
for every $A$-$\s$-algebra $R$. Clearly, $\gamma$ is an isomorphism of $G_A$-torsors. Note that $A\otimes \gamma \colon A \te Y \to A \te \tilde{X}_A$ is given by 
$ G(S) \to\tilde{X}(S), \ g \mapsto g.f_1(\tilde{x})$
for every $(A\te A)$-$\s$-algebra $S$ and similarly $\gamma \otimes A \colon Y \te A \to \tilde{X}_A \te A$ is given by 
$G(S) \to \tilde{X}(S), \ g \mapsto g.f_2(\tilde{x})$
for every $(A\te A)$-$\s$-algebra $S$. The condition $f_1(\tilde{x})=\chi. f_2(\tilde{x})$ thus implies that the diagram
\[\xymatrix{A\te Y \ar@{->}[d]_{A\otimes\gamma} \ar@{->}[rr]^{\ \phi \ } &&Y\te A \ar@{->}[d]^{\gamma \otimes A}\\
            A \te \tilde{X}_A \ar@{->}[rr]_{\tilde \phi} &&  \tilde{X}_A\te A} \]
commutes. Therefore, $\tilde{X}$ and $X$ are isomorphic as $G$-torsors and thus $\Upsilon\circ \Xi=\Id_{\Sigma_{A/k}(G)}$.

Conversely, let us start with a cocycle $\chi \in \operatorname{Z}^1(A/k,G)$ and set $Y=G_A$ with $A/k$-descent datum $\phi \colon A \te Y \to Y \te A$ given by multiplication with $\chi$ from the right. Consider $Y$ as a $G_A$-torsor via multiplication from the left. The second step implies that $Y$ descends to a $G$-torsor $X$. We claim that $x:=\chi^{-1} \in G(A\te A)=Y(A\te A)$ is contained in $X(A)$ and that $f_1(x)=\chi.f_2(x)$, where $f_1,f_2 \colon X(A) \to X(A\te A)$ are defined as in the first step. Consider the diagram 
\[\xymatrix{ G(A \te A \te A)  \ar@{->}[rr]^{\ \phi_{A\otimes  A \otimes A} \ } && G(A \te A \te A) \\
            & G(A\te A)  \ar@{->}[ul]^{\del_2} \ar@{->}[ur]_{\del_3} & } \] where we consider $A\te A \te A$ as an $(A\te A)$-algebra via $a\otimes b\mapsto 1\otimes a\otimes b$. In particular, $\phi_{A\otimes  A \otimes A}$ is given by multiplication from the right with $\del_1(\chi)$. 
Recall that 
\[X(A)= \{ z \in G(A\te A) \mid \phi_{A\otimes A\otimes A}(\del_2(z))=\del_3(z) \} \] (compare with the proof of Theorem \ref{abstieg}). We compute $$\phi_{A\otimes A\otimes A}(\del_2(x))=\del_2(x)\del_1(\chi)=\del_2(\chi)^{-1}\del_1(\chi)=\del_3(\chi)^{-1}=\del_3(x).$$ Hence $x=\chi^{-1}$ is contained in $X(A)$.

The action of $G$ on $X$ is defined via restriction of the $G_A$-action on $Y$. Thus for every $k$-$\s$-algebra $R$, $G(R)$ acts on $X(R)\subseteq G(R\te A)$ via multiplication inside $G(R\te A)$. In particular, $g.z=\del_3(g)\cdot z$ for every $g \in G(A\te A)$ and every $z \in X(A\te A)\subseteq G(A\te A\te A)$. Thus  $\chi.f_2(x)=\del_3(\chi)\cdot f_2(x)$ with multiplication in $G(A\te A\te A)$. Note that $f_1,f_2 \colon X(A) \to X(A\te A)$ are the restrictions of $\del_1,\del_2 \colon G(A\te A) \to G(A\te A \te A)$ from $G(A\te A)$ to $X(A) \subseteq G(A\te A)$. Thus 
\[f_1(x)=\del_1(x)=\del_1(\chi)^{-1}=\del_3(\chi)\del_2(\chi^{-1})=\del_3(\chi)f_2(x)=\chi.f_2(x)\]
 as claimed. We conclude that $\Xi \circ \Upsilon=\Id$.
 \end{proof}

In the following example we show how Theorem \ref{Klassifikation} can be used to determine all torsors for the $\s$-closed subgroups of $\Gm^n$.
\begin{ex}
	Let $k$ be a $\s$-field and let $G$ be a $\s$-closed subgroup of $\Gm^n$. 
	An element $f\in k\{\Gm^n\}= k\{y_1,y_1^{-1},y_2,y_2^{-1},\ldots,y_n,y_n^{-1}\}$ is called a \emph{multiplicative function} if
	$$f=y^{\alpha_0}\s(y^{\alpha_1})\cdots\s^l(y^{\alpha_l}) \text{ where } \alpha_1,\ldots,\alpha_l\in\Z^n$$
	and $y^\beta=y_1^{\beta_1}y_2^{\beta_2}\ldots y_n^{\beta_n}$ for $\beta=(\beta_1,\ldots,\beta_n)\in\Z^n$.
	In \cite[Lemma A.40]{DHW} it is shown that
	$G$ can be defined by using a set $F$ of multiplicative functions and it follows from \cite[Theorem 2.3.1]{Wibmer:Habil} that we may assume $F$ to be finite. So let us fix a finite set $F=\{f_1,\ldots,f_m\}$ of multiplicative functions such that
	$$G(R)=\{g\in (R^\times)^n|\ f_1(g)=1,\ldots,f_m(g)=1\}$$
	for any \ks-algebra $R$.
	We will show that every $G$-torsor $\tilde{X}$ is isomorphic to a $G$-torsor $X$ of the form
	\begin{equation} \label{eqn: Gmn torsor}
	X(R)=\{x\in (R^\times)^n|\ f_1(x)=a_1,\ldots,f_m(x)=a_m\}
	\end{equation}
	for any \ks-algebra $R$, where $a_1,\ldots,a_m\in k^\times$.
	
	Let $A$ be a \ks-algebra such that $\tilde{X}$ is trivial over $A$ and let $\chi\in G(A\otimes_k A)$ be a cocycle for $G$ corresponding to $\tilde{X}$ (as constructed in the proof of Theorem \ref{Klassifikation}). Since $\chi$ is trivial when considered as a cocylce for $\Gm^n$ (Example \ref{ex:  Gan and Gmn trivial}) there exists a $g\in \Gm^n(A)$ such that $\chi=\delta_1(g)\delta_2(g)^{-1}$.
	Since $\chi\in G(A\otimes_k A)$ we have
	$f_i(\chi)=1$ for $i=1,\ldots,m$. But 
	$$f_i(\chi)=f_i(\delta_1(g)\delta_2(g)^{-1})=f_i(g)\otimes (f_i(g))^{-1}\in A\otimes_k A.$$
	Thus $a_i=f_i(g)\in k$ for $i=1,\ldots,m$.
	
	Now let $X$ be the $G$-torsor as defined in Equation (\ref{eqn: Gmn torsor}) above. (Note that $X$ is not empty as $g\in X(A)$.) We claim that $X$ and $\tilde{X}$ are isomorphic. Indeed, the cocycle corresponding to $X$ (constructed from $g\in X(A)$ as in the proof of Theorem \ref{Klassifikation}) is $\chi$. Thus $X$ and $\tilde{X}$ are isomorphic by Theorem \ref{Klassifikation}.		
\end{ex}

Let $\G$ be an algebraic group over $k$ and let $\X$ be a $\G$-torsor. Then clearly the corresponding \ks-variety $X$ (i.e., $X(R)=\X(R^\sharp)$ for every \ks-algebra $R$) is a $G$-torsor for the corresponding $\s$-algebraic group $G$. The following corollary provides a converse.   

\begin{cor}
	Let $\G$ be an algebraic group over a $\s$-field $k$ and let $G$ be the corresponding $\s$-algebraic group. Then:
\begin{enumerate}
	\item Every $G$-torsor is isomorphic to a $G$-torsor $X$, corresponding to a $\G$-torsor $\X$.   
	\item Two $G$-torsors $X$ and $Y$ as in a) are isomorphic as (difference) $G$-torsors if and only if  $\X$ and $\Y$ are isomorphic as (algebraic) $\G$-torsors. 
\end{enumerate}

\end{cor}
\begin{proof}
Since every $G$-torsor is trivial over some \ks-algebra this is clear from Proposition \ref{algebr Gruppen} and Theorem \ref{Klassifikation}.
\end{proof}

The following lemma will be needed in Section \ref{sec: Isomorphismclasses}.

\begin{lem} \label{lem: embed into Gln}
	Let $k$ be a $\s$-field and let $G$ be a $\s$-closed subgroup of $\GL_n$. Then every $G$-torsor is isomorphic (as a $G$-torsor) to a $\s$-closed $\s$-subvariety $X$ of $\GL_n$, where the action of $G$ on $X$ is given by matrix multiplication.
\end{lem}
\begin{proof}
	Let $X'$ be a $G$-torsor and let $A$ be a \ks-algebra such that $X'$ is trivial over $A$. Let $\chi\in G(A\otimes_k A)$ be a cocycle corresponding to $X'$ (as constructed in the proof of Theorem \ref{Klassifikation}). As in the proof, we use $\chi$ to define a descent datum on $G_A$. Since $\chi$ can also be interpreted as a cocylce for $\GL_n$, it also defines a descent datum on $(\GL_n)_A$.
	The $\s$-closed embedding $G_A\to (\GL_n)_A$ is $G_A$-equivariant and compatible with the descent data. By Lemma \ref{lem: sclosed embedding descends} it therefore descends to a $G$-equivariant $\s$-closed embedding $X'\to X''$ where $X''$ is some $\GL_n$-torsor. Since every $\GL_n$-torsor is trivial (Corollary \ref{cor: alg Gr triviale H1} and Theorem \ref{Klassifikation}) $X''$ is isomorphic to $\GL_n$ with left multiplication.
	Composing  $X'\to X''$ with this isomorphism yields the desired embedding of $X'$ into $\GL_n$. 
\end{proof}

\section{ $\h^1(k,G)$ and the cohomology sequence} \label{sec 5}
In this section $k$ is a $\s$-field and $G$ is a $\s$-algebraic group over $k$. Note that because $k$ is a field any \ks-algebra $A$ is faithfully flat and therefore the results of the previous sections apply to all \ks-algebras.


\begin{lem}\label{lem: compatibility}
	Let $A \to B$ be a morphism of \ks-algebras. Then the diagram 
		\[\xymatrix{\Sigma_{A/k}(G) \ar@{^{(}->}[d] \ar@{->}[rr]^\sim && \h^1(A/k,G) \ar@{->}[d]\\
			\Sigma_{B/k}(G) \ar@{->}[rr]^\sim && \h^1(B/k,G)} \] 
		commutes, where the horizontal arrows are the isomorphisms from Theorem \ref{Klassifikation}, the left vertical arrow is the natural inclusion and the right vertical arrow is the canonical map $\h^1(A/k,G) \to \h^1(B/k,G)$. In particular, the latter map is injective and does not depend on the choice of the morphism $A \to B$.
\end{lem}
\begin{proof}
	Let $X$ be a $G$-torsor that is trivial over $A$, so there exists an $x \in X(A)$.
	Let us fix the following notation: we consider the canonical maps $s \colon X(A)\to X(B), \ r \colon X(A\te A) \to X(B\te B)$ and $u \colon G(A\te A)\to G(B\te B)$ induced from $A \to B$ and the canonical maps $f_1,f_2 \colon X(A)\to X(A\te A), g_1,g_2 \colon X(B) \to X(B\te B)$. 
	
	 Then $s(x) \in X(B)$, hence $X$ is trivial over $B$. Thus there is indeed an inclusion $\Sigma_{A/k}(G) \subseteq \Sigma_{B/k}(G)$. Let $\chi \in G(A\te A)$ be the unique element with $f_1(x)=\chi. f_2(x)$. By the proof of Theorem \ref{Klassifikation}, $\chi$ is a cocycle and its equivalence class is the image of the isomorphism class of $X$ under the upper horizontal arrow. The image of the equivalence class of $\chi$ under the right vertical arrow is represented by $u(\chi)$. It remains to show that the equivalence class represented by $u(\chi)$ corresponds to the isomorphism class of $X$ under the lower horizontal arrow. Note that $r(f_1(x))=r(\chi.f_2(x))=u(\chi).r(f_2(x))$. As $r\circ f_1=g_1\circ s$ and $r \circ f_2=g_2\circ s$, we conclude $g_1(s(x))=u(\chi).g_2(s(x))$ and the claim follows.
\end{proof}

Let $I$ denote the set of isomorphism classes of \ks-algebras. We define a preorder on $I$ as follows: If $A$ and $B$ are \ks-algebras representing equivalence classes $[A]$ and $[B]$ then $[A]\leq[B]$ if there exists a morphism of \ks-algebras $A\to B$. Since $[A],[B]\leq [A\otimes_k B]$ we see that $I$ is a directed set. 
If $[A]\leq [B]$, according to Lemma \ref{lem: compatibility}, we have a canonical (injective) morphism of pointed sets $d_{A,B}\colon \h^1(A/k,G)\to  \h^1(B/k,G)$.
If $[A]\leq [B]$ and $[B]\leq [C]$ it is clear from Lemma \ref{lem: compatibility} that $d_{A,C}=d_{B,C}\circ d_{A, B}$. We thus have a directed system of pointed sets and we can define $\h^1(k, G)$ as the direct limit
$$\h^1(k, G)=\varinjlim \h^1(A/k,G).$$
In general, $\h^1(k, G)$ is just a pointed set, but if $G$ is abelian, 
the maps $d_{A,B}$ are morphisms of abelian groups and so $\h^1(k, G)$ is also an abelian group. Since every $G$-torsor $X$ becomes trivial over some \ks-algebra $A$ (e.g., $A=k\{X\}$) we obtain from Theorem \ref{Klassifikation}:
\begin{thm} \label{thm: H1 bij torsors}
The cohomology set $\h^1(k, G)$ is in bijection with the isomorphism classes of $G$-torsors.	 \qed
\end{thm}

If $\phi\colon G\to H$ is a morphism of $\s$-algebraic groups and $A\to B$ a morphism of \ks-algebras, then

$$
\xymatrix{
\h^1(A/k,G) \ar[r] \ar[d] & \h^1(A/k,H) \ar[d] \\
\h^1(B/k,G) \ar[r] & \h^1(B/k,H)
}
$$
commutes and we obtain a map of pointed sets $\h^1(k,G)\to\h^1(k,H)$.

\begin{ex}\label{ex3}
	 As in Examples \ref{ex1} and \ref{ex2}, let $G$ be the $\s$-algebraic subgroup of $\Gm$ given by
	 $$G(R)=\{g\in R^\times|\ g^2
	=1,\ \s(g)=g\}$$ for any \ks-algebra $R$.
	Let $M$ be the set of all pairs $(a,b)\in k^\times \times k^\times$ such that $\s(a)=ab^2$. Define an equivalence relation on $M$ by $(a,b)\sim (a',b')$ if there exists a $\lambda\in k^\times$ such that $a'=\lambda^2a$ and $b'=\frac{\s(\lambda)}{\lambda}b$. Note that $M$ is an abelian group under multiplication and also
	$\mathtt{H}=M/\sim$ is an abelian group. We claim that
	\begin{enumerate}[(1)]
		\item $\h^1(k,G)\cong\mathtt{H}$.
		\item Every $G$-torsor is isomorphic to $X_{a,b}$	
		for some $a,b\in k^\times$ with $\s(a)=ab^2$, where
		$X_{a,b}$ is given by $X_{a,b}(R)=\{x\in R|\ x^2=a,\ \s(x)=bx\}$ for any \ks-algebra $R$.   
		\item Two $G$-torsors $X_{a,b}$ and $X_{a',b'}$ are isomorphic if and only if $(a,b)\sim(a',b')$.  
	\end{enumerate}
Let $A$ be a \ks-algebra.
	We saw in Example \ref{ex2} that
	\[ \h^1(A/k, G)\cong\{\alpha \in A^\times \ | \ \alpha^2 \in k, \ \sigma(\alpha)/\alpha \in k \}/\sim  \] where $\alpha \sim \alpha'$ if there is a $\lambda \in k^\times$ with $(\alpha')^2=\lambda^2 \alpha^2$ and $\sigma(\alpha')/\alpha'=\sigma(\lambda)/\lambda \cdot \sigma(\alpha)/\alpha$. 
	Setting $a=\alpha^2$ and $b=\s(\alpha)/\alpha$ yields a morphism of abelian groups $\h^1(A/k, G)\to\mathtt{H}$. If $[A]\leq[B]$, then
	$$
	\xymatrix{
		\h^1(A/k, G) \ar[rr] \ar[rd] & & \h^1(B/k, G) \ar[ld]  \\
		& \mathtt{H} &		
	}
	$$
	commutes and we obtain a morphism $\h^1(k, G)\to \mathtt{H}$.
	
	Conversely, let $(a,b)\in k^\times \times k^\times$ with $\s(a)=ab^2$. Then we can define a \ks-algebra structure on $A=k[y]/(y^2-a)$ by setting $\s(y)=by$. 
	If $(a,b)\sim (a',b')$ then $A\cong A'$ via $y\mapsto \frac{1}{\lambda}y'$. The image $\alpha$ of $y$ in $A$ satisfies $\alpha^2=a\in k$ and $\s(\alpha)/\alpha=b\in k$. Then $\chi=\delta_1(\alpha)\delta_2(\alpha)^{-1}\in\z^1(A/k,G)$ and since $\lambda\alpha\mapsto \alpha'$ under $A\cong A'$ we see that the image of $\chi$ in $\h^1(k, G)$ only depends on the equivalence class of $(a,b)$. So we have a map $\mathtt{H}\to \h^1(k, G)$ and it is clear that these two maps are inverse to each other, so (1) follows.
		
	To prove (2), let $X$ be a $G$-torsor and fix a $k$-$\s$-algebra $A$ such that $X$ is trivial over $A$. We already saw in Example \ref{ex2}, that a cocylce $\chi \in \z^1(A/k,G)$ corresponding to $X$ is of the form $\chi=\alpha^{-1}\otimes \alpha$ for some
	$\alpha \in A^\times$ with $\alpha^2, \sigma(\alpha)/\alpha \in k^\times$. Hence the torsor $X_{a,b}$ with $a=\alpha^2$ and $b=\sigma(\alpha)/\alpha$ is isomorphic to $X$, since $X$ and $X_{a,b}$ define the same equivalence class in $\h^1(A/k,G)$.	For (3), assume $(a,b)\sim(a',b')$ and let $\lambda\in k^\times$ be such that $a'=\lambda^2a$ and $b'=\frac{\s(\lambda)}{\lambda}b$. Then $x\mapsto\lambda x$ defines an isomorphism of $G$-torsors $X_{a,b}\cong X_{a',b'}$. Conversely, if $(a,b)$ and $(a',b')$ are not equivalent, then $X_{a,b}$ and $X_{a',b'}$ correspond to different elements in $\mathtt{H}\cong \h^1(k, G)$ and therefore are not isomorphic.
\end{ex}


\begin{ex} \label{ex Ga}
	Let $k$ be a $\s$-field and $\LL(y)=\s^n(y)+\lambda_{n-1}\s^{n-1}(y)+\ldots+\lambda_0y$ a linear difference equation over $k$. As in Examples \ref{ex: Ga intro} and \ref{ex Ga H1(A)}, we consider the $\s$-closed subgroup $G$ of $\Ga$ given by
	$$G(R)=\{g\in R \mid \LL(g)=0\}\subseteq \Ga(R)$$
	for any \ks-algebra $R$. Note that if $k$ has characteristic zero, then any $\s$-closed subgroup of $\Ga$ is of this form (\cite[Cor. A.3]{DHW2}).
	
	We define an equivalence relation on $k$ by $a\sim a'$ if there exists an element $b\in k$ such that $a'-a=\LL(b)$. 
	Then $\mathtt{H}=k/\sim =k/\LL(k)$ is an abelian group under addition and we claim that
	\begin{enumerate}[(1)]
		\item $\h^1(k,G)\cong\mathtt{H}.$ In particular, $\h^1(k,G)$ is trivial, if and only if $\LL\colon k\to k$ is surjective.
			\item Every $G$-torsor is isomorphic to $X_a$ for some $a\in k$,  where $X_a$ is defined by
			$X_a(R)=\{x\in R|\ \LL(x)=a\}$ for any \ks-algebra $R$. 
			\item Two $G$-torsors $X_a$ and $X_{a'}$ are isomorphic if and only if $a\sim a'$.
	\end{enumerate}
	
	Let $A$ be a \ks-algebra. We saw in Example \ref{ex Ga H1(A)} that
		$$\h^1(A/k,G)\cong \{\alpha\in A \ | \ \LL(\alpha)\in k\}/\sim,$$
		where $\alpha\sim \alpha'$ if $\LL(\alpha'-\alpha)\in\LL(k)$.	Setting $a=\LL(\alpha)$ defines a morphism of abelian groups $\h^1(A/k,G)\to\mathtt{H}$ and this induces  a morphism of abelian groups $\h^1(k, G)\to \mathtt{H}$.
	 
	 Conversely, for $a\in k$ let $A$ be a \ks-algebra such that there exists an $\alpha\in A$ with $\LL(\alpha)=a$. (For example, we may take $A=k[y_1,\ldots,y_n]$ with $\s(y_1)=y_2,  \s(y_2)=y_3,\ldots,\s(y_n)=-\lambda_{n-1}y_n-\ldots-\lambda_0y_1$ and $\alpha=y_1\in A$.)
Then $\chi=1\otimes\alpha-\alpha\otimes 1$ is a cocylce for $G$. The image of $\chi$ in $\h^1(k, G)$ only depends on the equivalence class of $a$, and the induced mapping $\mathtt{H}\to \h^1(k, G)$ is the inverse to the mapping constructed previously. Hence (1) follows.

For (2), let $X$ be a $G$-torsor and $A$ a \ks-algebra such that $X$ is trivial over $A$. We already saw in Example \ref{ex Ga H1(A)} that a cocycle $\chi\in\operatorname{Z}_\s^1(A/k,G)$ corresponding to $X$ is of the form $\chi=1\otimes\alpha-\alpha\otimes 1$ for some $\alpha\in A$ with $\LL(\alpha)\in k$. Set $a=\LL(\alpha)$. Since $X_a$ and $X$ correspond to the same element in $\h^1(A/k,G)$, we see that $X_a$ and $X$ are isomorphic.
	
%
	
Finally, let $a,a'\in k$ such that there exists an element $b\in k$ with $a'-a=\LL(b)$. Then
	$x\mapsto x+b$ defines an isomorphism $X_{a}\to X_{a'}$ of $G$-torsors. Conversely, if $a$ and $a'$ are not equivalent, then the torsors $X_a$ and $X_{a'}$ correspond to different elements in
	$\mathtt{H}\cong\h^1(k,G)$ and therefore are not isomorphic. Hence (3) follows.
	
%
%
\end{ex}
From Corollary \ref{cor: alg Gr triviale H1} and Example \ref{ex:  Gan and Gmn trivial} we obtain:
\begin{ex} \label{ex: h1 trivial for alg groups}
	The cohomology set $\h^1(k, G)$ is trivial for all $\s$-algebraic groups $G$ corresponding to the algebraic groups $\Gm^n$, $\Ga^n$, $\GL_n$ and $\SL_n$.
\end{ex}

\medskip

To establish the long exact cohomology sequence we need to know a few things about quotients of $\s$-algebraic groups. For a morphism $\phi\colon G\to H$ of $\s$-algebraic groups we set $\ker(\phi)(R)=\ker(\phi_R)$ for any \ks-algebra $R$.


If $N$ is a normal $\s$-closed subgroup of $G$ there exists a morphism of $\s$-algebraic groups $\pi\colon G\to G/N$ satisfying the following universal property: $N\subseteq\ker(\pi)$ and for any morphism of $\s$-algebraic groups $\phi\colon G\to H$ such that $N\subseteq\ker(\phi)$ there exists a unique morphism $G/N\to H$ such that 
$$
\xymatrix{
G \ar^{\pi}[rr] \ar_\phi[rd]& & G/N \ar@{..>}[ld] \\
 & H	&
}
$$  
commutes. Moreover, $\ker(\pi)=N$. (See \cite[Theorem 3.2.3]{Wibmer:Habil}.) The morphism $\pi\colon G\to G/N$ is in a certain sense surjective (see \cite[Theorem 3.3.6]{Wibmer:Habil} for more details), but it is not surjective in the naive sense that $\pi_R\colon G(R)\to (G/N)(R)$ is surjective for all \ks-algebras $R$. However, 
what is true, is that for every \ks-algebra $R$ and every $h\in (G/N)(R)$ there exists an $R$-$\s$-algebra $S$ 
and a $g\in G(S)$ such that $\pi(g)\in (G/N)(S)$ equals the image of $h$ under $(G/N)(R)\to (G/N)(S)$.



Recall that a sequence of pointed sets $M'\to M\to M''$ is exact (at $M$) if the inverse image of the distinguished element of $M''$ equals the image of $M'$ in $M$.

\begin{prop}\label{exact sequence}
	Let $N$ be a normal $\s$-closed subgroup of $G$. Then there exists a morphism of pointed sets $\delta \colon (G/N)(k) \to \h^1(k,N)$ such that the sequence
	$$1 \to N(k) \to G(k) \to (G/N)(k) \rightarrow^{\! \! \! \! \! \! \delta} \ \h^1(k,N) \to \h^1(k,G) \to \h^1(k,G/N)$$
	 of pointed sets is exact.
\end{prop}
\begin{proof}
	To define $\delta$, let $x \in (G/N)(k)$. Then there exists a $k$-$\s$-algebra $A$ and an element $g \in G(A)$ such that $\pi(g)=x$ (more precisely, $\pi$ maps $g$ to the image of $x$ in $(G/N)(A)$). Define $\chi=\delta_1(g)\delta_2(g)^{-1} \in G(A\te A)$. As $x$ is $k$-rational, $\pi(\chi)=1$. Hence $\chi \in N(A\te A)$ (by \cite[Thm. 7.3]{wib}) and $\chi$ is a cocycle for $N$. We would like to define $\delta(x)\in\h^1(k,G)$ as the image of the equivalence class of $\chi$ but we first need to show that this is independent of the choice of $A$ and $g$. Let $B$ be another $k$-$\s$-algebra such that there exists an element $h \in G(B)$ with $\pi(h)=x$ and define $\mu=\delta_1(h)\delta_2(h)^{-1} \in N(B\te B)$.
	
	Let $C$ be a \ks-algebra with $[C]\geq [A],[B]$ and fix morphisms $A\to C$ and $B\to C$. We will show that $\chi$ and $\mu$ have the same image in $\h^1(C/k,N)$. (So that we have a well-defined map $\delta$.)
	 We consider the canonical maps $w \colon G(A) \to G(C)$, $t \colon G(B) \to G(C)$, $u \colon G(A\te A) \to G(C\te C)$ and $v \colon G(B\te B) \to G(C\te C)$. It suffices to show that $u(\chi)$ and $v(\mu)$ are equivalent. Since $x$ is $k$-rational, $\pi(w(g))=\pi(t(h))$ and so there exists an $n \in N(C)$ such that $w(g)=nt(h)$. Thus $u(\chi)=\delta_1(w(g))\delta_2(w(g))^{-1}=\delta_1(n)\delta_1(t(h))\delta_2(t(h))^{-1}\delta_2(n)^{-1}$ is equivalent to $\delta_1(t(h))\delta_2(t(h))^{-1}=v(\mu)$. 
	
	The exactness at $N(k)$ is trivial and the exactness at $G(k)$ follows from $\ker(\pi)=N$. To prove exactness at $(G/N)(k)$, let $x \in (G/N)(k)$ be such that $\delta(x)$ is trivial. Fix a $k$-$\s$-algebra $A$ such that there is a $g \in G(A)$ with $\pi(g)=x$. As $\delta(x)$ is trivial, there exists an $n \in N(A)$ with $\delta_1(g)\delta_2(g)^{-1}=\delta_1(n)\delta_2(n)^{-1}$. Thus $n^{-1}g$ is contained in $G(k)$ and $\pi(n^{-1}g)=x$. 
	
	To prove exactness at $\h^1(k,N)$, let $A$ be a $k$-$\s$-algebra and let $\chi \in \z^1(A/k,N)$ be a cocycle whose class is trivial in $\z^1(A/k,G)$. Thus there exists a $g \in G(A)$ with $\chi=\delta_1(g)\delta_2(g)^{-1}$. Define $x=\pi(g) \in (G/N)(A)$. As $\chi \in N(A\te A)$, $\pi(\chi)=1$ and so $\delta_1(\pi(g))=\delta_2(\pi(g))$. Thus $x \in (G/N)(k)$ and the claim follows. 
	
	To prove exactness at $\h^1(k,G)$ let $A$ be a $k$-$\s$-algebra and let $\chi \in \z^1(A/k,G)$ be a cocycle whose class maps to the trivial class under $\h^1(k,G)\to\h^1(k,G/N)$.
	So there exists an $x \in (G/N)(A)$ such that $\pi(\chi)=\delta_1(x)\delta_2(x)^{-1}$. Let $B$ be an $A$-$\s$-algebra such that there exists a $g \in G(B)$ with $\pi(g)=x$. Consider the canonical map $u \colon G(A\te A)\to G(B\te B)$. Then $\pi(u(\chi))=\pi(\delta_1(g)\delta_2(g)^{-1})$, so $\delta_1(g)^{-1}u(\chi)\delta_2(g)$ is contained in $N(B\te B)$ and thus defines a cocycle in $\z^1(B/k,N)$. The image of $\delta_1(g)^{-1}u(\chi)\delta_2(g)$ in $\h^1(k, N)$ maps to the image of $\chi$ in $\h^1(k,G)$.
\end{proof}

\begin{ex} \label{ex exakte sequenz}
Let $G$ be a $\s$-algebraic group and let $d\geq 1$. Note that $\s^d\colon k\to k$ is a morphism of $\s$-fields. Thus the base extension ${^{\s^d}}\! G$ of $G$ via $\s^d\colon k\to k$ is a $\s$-algebraic group over $k$. Moreover, we have a morphism of $\s$-algebraic groups $\s^d\colon G\to {^{\s^d}}\! G$ defined as follows: For a \ks-algebra $R$ let ${_{\s^d} R}$ denote the \ks-algebra whose underlying difference ring is $R$ with $k$-algebra structure given by $k\xrightarrow{\s^d} k\to R$. Then ${^{\s^d}}\! G(R)=G({_{\s^d} R})$ and $\s^d\colon R\to {_{\s^d} R}$ is a morphism of \ks-algebras. We thus have a morphism of groups $\s^d\colon G(R)\to {^{\s^d}}\! G(R)$. Let $N=\{g\in G \mid \s^d(g)=1\}$ be the kernel of $\s^d\colon G\to {^{\s^d}}\! G$. The morphism  $\s^d\colon G\to {^{\s^d}}\! G$ is surjective if and only if $G$ is absolutely $\s$-reduced, i.e., $\s\colon k\{G\}\otimes_k K\to k\{G\}\otimes_k K$ is injective for all $\s$-field extensions $K$ of $k$, or equivalently, $k\{G\}$ is a $\s$-separable \ks-algebra. (See \cite[Section 6.3.4]{Wibmer:Habil} and \cite[Section 1.1]{TomasicWibmer:Stronglysetal} for more details.)
Let us assume that $G$ is absolutely $\s$-reduced so that $\s^d\colon G\to {^{\s^d}}\! G$ is the quotient of $G$ mod $N$.

We will show that $\h^1(k,N)$ is trivial if $\h^1(k,G)$ is trivial and $\s\colon k\to k$ is an automorphism.
According to Proposition \ref{exact sequence} we have an exact sequence $$G(k)\xrightarrow{\s^d}{^{\s^d}\! G}(k)\xrightarrow{\de}\h^1(k,N)\to \h^1(k,G).$$ But since $\s\colon k\to k$ is bijective, also $G(k)\xrightarrow{\s^d}{^{\s^d}\! G}(k)$ is bijective and it follows that $\h^1(k,N)$ is trivial.

If $\s\colon k\to k$ is not an automorphism, $\h^1(k,N)$ need not be trivial. For example, if $G=\GL_n$, then for a matrix $a\in\GL_n(k)$ with $a\notin\GL_n(\s(k))$, the $N$-torsor $X$ given by $X(R)=\{x\in\GL_n(R)|\ \s^d(x)=a\}$ for any \ks-algebra $R$ is non-trivial.
\end{ex}

\section{Classification results}\label{C}

In this section we use Theorem \ref{Klassifikation} to determine all torsors for a large class of $\s$-algebraic groups.


Throughout this section $k$ is a $\s$-field.
%
%
Let $\X$ be an affine scheme of finite type over $k^\sharp$. To simplify the notation we will drop the $\sharp$ in the sequel, e.g., for a \ks-algebra $R$ we will write $\X(R)$ instead of $\X(R^\sharp)$.


For $d\geq 1$ let $^{\s^d \! \!}\mathcal{X}$ be the affine scheme of finite type over $k$ obtained from $\mathcal{X}$ by base extension via $\sigma^d \colon k \to k$. Similarly, if $\varphi \colon \X \to \mathcal{Y}$ is a morphism of affine schemes of finite type over $k$, then we let $^{\s^d \! \!}\varphi \colon \Xd \to ^{\s^d \! \!}\mathcal{Y}$ be the morphism obtained from $\varphi$ by base extension via $\s^d \colon k \to k$. More explicitly, $\Xd(S)=\X(_{\s^d \!} S)$ for every $k$-algebra $S$, where $_{\s^d \!} S$ denotes the ring $S$ considered as a $k$-algebra via $k \longrightarrow^{\! \! \! \! \! \! \! \! \! \s^d} \ k \to S$.
Note that $(^{\s^d\! \!}\varphi)_S=\varphi_{\left( _{\s^d \!} S\right)}\colon \X(_{\s^d \!} S) \to \mathcal{Y}(_{\s^d \!} S)$.

If $\G$ is an algebraic group over $k$ (i.e., an affine group scheme of finite type over $k$) and $\X$ is a $\G$-torsor, then $\Gd$ is an algebraic group over $k$ and $\Xd$ is a $\Gd$-torsor via the base extension $\Gd \times \Xd \to \Xd$ of $\G \times \X \to \X$.

For every $k$-$\s$-algebra $R$ the map $\s^d \colon R \to \Rd$ is a morphism of $k$-algebras and therefore induces a map $\s^d \colon \X(R) \to {\Xd(R)}$.

\begin{rem}\label{remsigma}
 If $\theta \colon \X \to \Y$ is a morphism of affine schemes of finite type over $k$, then the following equation holds for all $k$-$\s$-algebras $R$ and all $x \in \X(R)$: $\s^d(\theta(x))={^{\s^d \! \!}\theta}(\s^d(x))$.
\end{rem}

\begin{thm}\label{UG mit psi}
Let $\G$ be an algebraic group over $k$ and for some fixed  $d\geq 1$ let $\psi \colon \G \to \Gd$ be a morphism of algebraic groups. We consider the $\s$-algebraic group $G$ with 
\[ G(R)=\{ g \in \G(R) \ | \ \s^d(g)=\psi(g) \} \] for all $k$-$\s$-algebras $R$. The following holds:
\begin{enumerate}
\item Let $\X$ be a $\G$-torsor and let $\varphi \colon \X \to \Xd$ be a morphism such that $\varphi(g.x)=\psi(g).\varphi(x)$ for all $k$-algebras $S$ and all $g \in \G(S), x \in \X(S)$. Then 
 \[ Y(R)=\{ x \in \X(R) \ | \ \s^d(x)=\varphi(x) \} \] for all $k$-$\s$-algebras $R$ defines a $G$-torsor $Y$.    
 \item Every $G$-torsor $X$ is isomorphic to a $G$-torsor $Y$ as defined in a). 
\item Two such $G$-torsors $Y$ and $\tilde Y$ given by $Y(R)=\{ x \in \X(R) \ | \ \s^d(x)=\varphi(x) \}$ and $\tilde Y(R)= \{ x \in \tilde \X(R) \ | \ \s^d(x)=\tilde \varphi(x) \}$ are isomorphic if and only if there exists an isomorphism of $\G$-torsors $\Theta \colon \X \to \tilde \X$ such that $^{\s^d \! \! }\Theta\circ \varphi=\tilde \varphi \circ \Theta$.
\end{enumerate}

\end{thm}
\begin{proof}

For a), it is easy to check that the action $\G \times \X \to \X$ restricts to $G \times Y \to Y$ and that $Y$ is a $G$-torsor.

To prove b), we fix a $k$-$\s$-algebra $A$ with $X(A) \neq \emptyset$ and let $[\chi] \in \h^1(A/k,G)$ be the equivalence class that corresponds to the isomorphism class of $X$ (use Theorem \ref{Klassifikation}). As $G(R) \subseteq \G(R)$ for all $k$-$\s$-algebras $R$, there is a canonical map $\h^1(A/k,G)\to \operatorname{H}^1(A/k, \G)$ and the image of $[\chi]$ thus corresponds to a $\G$-torsor $\X$ that is also trivial over $A$ (see for example \cite[Chapter III, \S4, 6.5]{DG}). Recall that $\chi \in G(A\te A)$ and that up to an isomorphism of torsors,
\begin{eqnarray*}
\X(S)&=& \{ g \in \G(S\te A) \ | \ \psi_2(g)=\psi_1(g)\cdot\chi \} \\ 
\Xd(S)&=& \X(_{\s^d \!} S)=\{ g \in \G(_{\s^d \!} S\te A) \ | \ \phi_2(g)=\phi_1(g)\cdot\chi \}
\end{eqnarray*} for all $k$-algebras $S$,
where $\psi_1,\psi_2$ denote the two canonical maps $\G(S\te A) \to \G(S\te A \te A)$ and $\phi_1,\phi_2 \colon \G(_{\s^d \!}S\te A) \to \G(_{\s^d \!}S\te A \te A)$. As $\X$ corresponds to the cocycle $\chi$, there exists an $x \in \X(A)$ with $f_1(x)=\chi.f_2(x)$, where $f_1,f_2 \colon \X(A) \to \X(A\te A)$. Let $y=\s^d(x) \in \Xd(A)$ and compute $g_1(y)=\s^d(\chi).g_2(y)$, where $g_1,g_2 \colon \Xd(A) \to \Xd(A\te A)$. As $\Xd$ is a $\Gd$-torsor, we conclude that the isomorphism class of $\Xd$ corresponds to $[\s^d(\chi)] \in \operatorname{H}^1(A/k, \Gd)$. Thus $\Xd$ is isomorphic to the $\Gd$-torsor $\Y$ defined by 
\[\Y(S)=\{g \in \G(_{\s^d \!}S\otimes_k {_{\s^d \!}A}) \mid \gamma_2(g)=\gamma_1(g)\cdot \s^d(\chi) \}, \] where $\gamma_1,\gamma_2 \colon \G(_{\s^d \!}S\otimes_k {_{\s^d \!}A})\to \G(_{\s^d \!}S\otimes_k {_{\s^d \!}A}\otimes_k {_{\s^d \!}A})$. We can also construct an explicit isomorphism $\gamma \colon \Xd \to \Y$: Let $S$ be a $k$-algebra. Then $s\otimes a \to s \otimes {\s^d(a)}$ defines a morphism of $k$-algebras $_{\s^d \!}S\otimes_k A \to \Sd\otimes_k {_{\s^d \!}A}$ which induces a group homomorphism $\gamma \colon \G(\Sd\otimes_k A) \to \G(\Sd\otimes_k {_{\s^d \!}A})$. It is easy to check that $\gamma$ restricts to a map $\gamma_S \colon \Xd(S) \to \Y(S)$ and that these maps $(\gamma_S)_S$ define a morphism of functors $\gamma \colon \Xd \to \Y$. Moreover, $\gamma$ is a morphism of $\Gd$-torsors and it follows that $\gamma$ is in fact an isomorphism.

We continue with the construction of a morphism $\varphi \colon \X \to \Xd$. It is easy to check that $\psi_{S\otimes A} \colon \G(S\te A) \to \Gd(S\te A)=\G(_{\s^d \!}S\otimes_k {_{\s^d \!}A})$ restricts to a map $\rho_{S} \colon \X(S) \to \Y(S)$ for every $k$-algebra $S$ and that this defines a morphism of functors $\rho \colon \X \to \Y$. Let $S$ be a $k$-algebra and let $g \in \G(S)$ and $x \in \X(S) \subseteq \G(S\te A)$. Recall that $g.x \in \X(S)$ is obtained from the multiplication of $g$ with $x$ inside $\G(S\te A)$ (compare with the second step in the proof of Theorem \ref{Klassifikation}). Thus $\rho(g.x)=\psi(g\cdot x)=\psi(g)\cdot \psi(x)=\psi(g).\rho(x)$. Let $\varphi \colon \X \to \Xd$ denote the composition $\varphi=\gamma^{-1}\circ\rho$. Then $\varphi(g.x)=\psi(g).\varphi(x)$ for all $k$-algebras $S$ and all $g \in \G(S), x \in \X(S)$. The isomorphism class of $X$ corresponds to $[\chi] \in \h^1(A/k,G)$ and thus $X$ is isomorphic to the torsor $Y$ with 
$ Y(R)= \{ g \in G(R\te A) \mid \epsilon_2(g)=\epsilon_1(g)\cdot \chi \}$ for all $k$-$\s$-algebras $R$, where $\epsilon_1,\epsilon_2$ denote the canonical maps $G(R\te A)\to G(R\te A \te A)$. We conclude 
\begin{eqnarray*}
  Y(R)&=& \{ g \in G(R\te A) \mid \epsilon_2(g)=\epsilon_1(g)\cdot \chi \} \\
&=& \{ g \in \G(R\te A) \mid \s^d(g)=\psi(g) \text{ and } \psi_2(g)=\psi_1(g)\cdot \chi\} \\
&=& \{ g \in \X(R) \mid \s^d(g)=\varphi(g)\},
\end{eqnarray*} where $\psi_1,\psi_2 \colon \G(R\te A ) \to \G(R\te A \te A)$. The claim follows. 

To prove part c), let us first assume that there exists an isomorphism of $\G$-torsors $\Theta \colon \X \to \tilde \X$ with $^{\s^d \! \! }\Theta\circ \varphi=\tilde \varphi \circ \Theta$. We claim that $\Theta$ restricts to an isomorphism $\theta \colon Y \to \tilde Y$. Let $R$ be a $k$-$\s$-algebra. Then for $x \in Y(R)$, we use Remark \ref{remsigma} to obtain $\sigma^d(\Theta(x))={^{\s^d \! \! }\Theta}(\s^d(x))={^{\s^d \! \! }\Theta}(\varphi(x))=\tilde \varphi(\Theta(x))$ and thus $\Theta(x) \in \tilde Y(R)$. Conversely, let $y \in \tilde Y(R)$ and let $x \in \X(R)$ be the unique element with $\Theta(x)=y$. Then $\s^d(y)=\tilde \varphi(y)$ implies ${^{\s^d \! \! }\Theta}(\s^d(x))={^{\s^d \! \! }\Theta}(\varphi(x))$ and as ${^{\s^d \! \! }\Theta}$ is an isomorphism we conclude $\s^d(x)=\varphi(x)$. Thus $x \in Y(R)$ and the claim follows. 

Conversely, assume that $Y$ and $\tilde Y$ are isomorphic as $G$-torsors. Fix a $k$-$\s$-algebra $A$ such that $Y$ and $\tilde Y$ are trivial over $A$. Let further $\chi\in \z^1(A/k,G)$ be a cocycle such that its equivalence class corresponds to the isomorphism class of $Y$ and $\tilde Y$. Then there exist elements $y \in Y(A)$ and $\tilde y \in \tilde Y(A)$ such that $f_1(y)=\chi.f_2(y)$ and $g_1(\tilde y)=\chi.g_2(\tilde y)$, where $f_1,f_2 \colon Y(A)\to Y(A\te A)$ and $g_1,g_2 \colon \tilde Y(A) \to \tilde Y(A\te A)$. Consider $\chi$ as a cocycle in $\operatorname{Z}^1(A/k,\G)$. Since $Y(A)\subseteq \X(A)$ and $\tilde Y(A) \subseteq \tilde \X(A)$, the isomorphism classes of $\X$ and $\tilde \X$ both correspond to the equivalence class $[\chi] \in \operatorname{H}^1(A/k,\G)$. Thus $\X$ and $\tilde \X$ are isomorphic as $\G$-torsors (see \cite[Chapter III, \S 4, 6.5]{DG}) and moreover, they are both isomorphic to the $\G$-torsor $\X'$ with 
\[\X'(S)= \{g \in \G(S\te A) \mid \psi_2(g)=\psi_1(g)\cdot \chi \} \] for all $k$-algebras $S$, where $\psi_1,\psi_2 \colon \G(S\te A) \to \G(S\te A \te A)$. More explicitly, there are isomorphisms $\vartheta \colon \X' \to \X$ and $\tilde \vartheta \colon \X'\to \tilde \X$ obtained from descending the isomorphisms $\G_A\to \X_A$ and $\G_A\to \tilde \X_A$ that are given by $g \mapsto g.y$ and $g \mapsto g.\tilde y$  (compare with the third step of the proof of Theorem \ref{Klassifikation}). 
We claim that $\Theta=\tilde \vartheta \circ \vartheta^{-1}$ has the desired property $^{\s^d \! \! }\Theta\circ \varphi=\tilde \varphi \circ \Theta$. It suffices to show that this holds after base extension from $k$ to $A$, so we claim that \[^{\s^d \! \! }\tilde \vartheta_A \circ ^{\s^d \! \! }\vartheta^{-1}_A\circ \varphi_A=\tilde \varphi_A \circ \tilde \vartheta_A \circ \vartheta^{-1}_A.\]
We use $(^{\s^d \! \! }\vartheta)_A= {^{\s^d \! \! }(\vartheta_A)}$ to obtain that $^{\s^d \! \! }\vartheta_A\colon \Gd_A \to \Xd_{\! A}$ is given by $g \mapsto g.\s^d(y)$ and similarly, $^{\s^d \! \! }\tilde \vartheta_A\colon \Gd_A \to ^{\s^d \! \! \! \!}\tilde \X_A$ is given by $g \mapsto g.\s^d(\tilde y)$. Let $S$ be an $A$-algebra and $z \in \X(S)$. Then there exists a $g \in \G(S)$ with $z=g.y$. We use $y \in Y(A)$, $\tilde y \in \tilde Y(A)$ to compute
\begin{eqnarray*}
 ^{\s^d \! \! }\tilde \vartheta(^{\s^d \! \! }\vartheta^{-1}(\varphi(g.y)))&=&^{\s^d \! \! } \tilde \vartheta(^{\s^d \! \! }\vartheta^{-1}(\psi(g).\varphi(y)))\\
&=&^{\s^d \! \! }\tilde \vartheta(^{\s^d \! \! }\vartheta^{-1}(\psi(g).\s^d(y)))\\
&=& \psi(g).\s^d(\tilde y)\\
&=&\psi(g).\tilde \varphi(\tilde y) \\
&=&\tilde \varphi(g.\tilde y) \\
&=&\tilde \varphi(\tilde \vartheta(\vartheta^{-1}(g.y))) 
\end{eqnarray*} and the claim follows.
 \end{proof}

\begin{cor} \label{cor:UG psi H1 trivial}
As in Theorem \ref{UG mit psi} we consider the $\s$-algebraic group $G$ given by
\[ G(R)=\{ g \in \G(R) \ | \ \s^d(g)=\psi(g) \}\]
for all \ks-algebras $R$. Let $X$ be a $G$-torsor and assume\footnote{Note that this assumption is satisfied if $k$ is algebraically closed.} that $\operatorname{H}^1(k, \G)=\{0\}$. Then there exists an element $a\in\Gd(k)$ such that $X$ is isomorphic to the $G$-torsor $Y$ given by
$$Y(R)=\{x\in\G(R)\mid \s^d(x)=\psi(x)a\}$$
for all \ks-algebras $R$.
\end{cor}
\begin{proof}
	According to Theorem \ref{UG mit psi} the $G$-torsor $X$ is isomorphic to $Y$ given by
	$$Y(R)=\{x\in\G(R)\mid  \s^d(x)=\varphi(x)\}$$
	for all \ks-algebras $R$, where $\varphi\colon\G\to \Gd$ is a morphism of schemes over $k$ such that $\varphi(g.x)=\psi(g).\varphi(x)$ for all $k$-algebras $S$ and all $g,x \in \G(S)$.
	So if $a=\varphi(1)\in\Gd(k)$ then $\varphi(x)=\psi(x)a$ for all \ks-algebras $R$ and $x\in\G(R)$. 
		\end{proof}
From Corollary \ref{cor:UG psi H1 trivial} we obtain a classification of torsors for (special) unitary groups:

\begin{ex}
	Let $\G=\GL_n$ or $\G=\SL_n$. If
$$G=\{ g \in \G \mid \s^d(g)=(g^{\operatorname{T}})^{-1} \},$$
then for every $G$-torsor $X$, there exists an element $a \in \G(k)$ such that $X$ is isomorphic to the $G$-torsor 
\[Y=\{g \in \G \mid \s^d(g)=(g^{\operatorname{T}})^{-1} a \}. \] 
Similarly, if 
	$$G=\{ g \in \G \mid \s^d(g)=g \},$$
	then for every $G$-torsor $X$, there exists an element $a \in \G(k)$ such that $X$ is isomorphic to the $G$-torsor 
	\[Y=\{g \in \G \mid \s^d(g)=g a \}. \] 
\end{ex}

\begin{cor}\label{cor}
Let $\G$ be a linear algebraic group over $k$ and consider the $\s$-algebraic group $G$ with 
\[ G(R)=\{g \in \G(R) \mid \s^d(g)=1 \} \] for all $k$-$\s$-algebras $R$. Then for every $G$-torsor $X$, there exists a $\G$-torsor $\X$ and an element $a \in \Xd(k)$, such that $X$ is isomorphic to the $G$-torsor $Y$ with 
\[Y(R)=\{ x \in \X(R) \mid \s^d(x)=a \} \] for every $k$-$\s$-algebra $R$. 
\end{cor}
\begin{proof}
If we let $\psi \colon \G \to \Gd$ be the morphism given by $g \mapsto 1$, then $G$ is defined as in Theorem \ref{UG mit psi}. Let $X$ be a $G$-torsor. By part b) of this theorem, there exists a $\G$-torsor $\X$ and a morphism $\varphi \colon \X \to \Xd$ with $\varphi(g.x)=\varphi(x)$ (for all $x \in \X(S)$, $g \in \G(S)$ and all $k$-algebras $S$) such that $X$ is isomorphic to the $G$-torsor $Y$ with $Y(R)=\{ x \in \X(R) \mid \s^d(x)=\varphi(x) \}$ for all $k$-$\s$-algebras $R$. Note that $\varphi_S \colon \X(S) \to \Xd(S)$ is a constant map for every $k$-algebra $S$, i.e., there exists an element $a_S \in \Xd(S)$ with $\varphi(x)=a_S$ for all $x \in \X(S)$. Since $\varphi$ is a morphism, $a_S$ maps to $a_{S\te S}$ under both maps $\Xd(S) \to \Xd(S\te S)$. As $\Xd$ is a sheaf, we obtain that $a_S \in \Xd(k)$. Thus there is an element $a \in \Xd(k)$ such that $\varphi_S(x)=a$ for all $x \in \X(S)$ and for every $k$-algebra $S$.  (In particular, the $\Gd$-torsor $\Xd$ is trivial.)
\end{proof}

\begin{cor} \label{cor: sg one}
Under the assumptions of Corollary \ref{cor}, assume in addition that $\s$ is an automorphism on $k$. Then every $G$-torsor is trivial (i.e., $\h^1(k,G)=\{0\}$).
\end{cor}

\begin{proof}
Let $X$ be a $G$-torsor. By Corollary \ref{cor}, there exists a $\G$-torsor $\X$ and an element $a \in \Xd(k)$ such that $X$ is isomorphic to the $G$-torsor $Y$ with $Y(R)=\{ x \in \X(R) \mid \s^d(x)=a \}$ for every $k$-$\s$-algebra $R$.
Since $\s^d \colon k \to k$ is a bijection also 
$\s^d \colon \X(k) \to \Xd(k)$ is bijective.
This shows that $Y(k)\neq\emptyset$. Thus $Y$ is trivial.
\end{proof}

Corollary \ref{cor: sg one} is not valid without the assumption that $\s\colon k\to k$ is surjective. For example, if we take $\G=\Gm$ and  $a\in k\smallsetminus\s(k)$, then $X=\{x\in\Gm\mid \s^d(x)=a\}$ is a non-trivial torsor for $G=\{g\in\Gm\mid \s(g)=1\}$.

\section{Isomorphism classes of $\s$-Picard-Vessiot rings} \label{sec: Isomorphismclasses}


Difference algebraic groups occur naturally as the Galois groups of linear differential equations depending on a discrete parameter. See \cite{DHW} and \cite{DHW2}. A similar Galois theory also exists for linear difference equations (\cite{OchinnikovWibmer:sGaloistheory}).
In this last section we present an application to this Galois theory. We show that $\h^1(k,G)$ classifies isomorphism classes of $\s$-Picard-Vessiot rings. This generalizes a result proved in the Tannakian context in \cite[Theorem 3.2]{DeligneMilne:TannakianCategories}. (See also \cite[Cor. 3.2]{Dyc}.)

Before recalling the basics of the $\s$-Picard-Vessiot theory we discuss a result that only pertains to $G$-torsors and will be needed in the proofs later.
In the following discussion, let $k$ be a $\s$-field and let $G$ be a $\s$-algebraic group over $k$.
%
%
Let $X$ be a \ks-variety equipped with a (left) $G$-action, i.e., a morphism of $\s$-varieties $G\times X\to X$ such that $G(R)\times X(R)\to X(R)$ is a group action for every \ks-algebra $R$. We will define a functorial (left) action of $G$ on $k\{X\}$. For every \ks-algebra $R$ the group $G(R)$ acts on $k\{X\}\otimes_k R$
$$G(R)\times (k\{X\}\otimes_k R)\to k\{X\}\otimes_k R, \ (g,f)\mapsto g(f)$$
as follows. The element $g\in G(R)$ defines an automorphism of $X_R$ by
$g\colon X_R\to X_R,\ x\mapsto g^{-1}.x$ and the dual map
$k\{X\}\otimes_k R\to k\{X\}\otimes_k R,\ f\mapsto g(f)$ is an $R$-$\s$-algebra automorphism.

%
An element $f\in k\{X\}$ is called \emph{$G$-invariant} if $g(f\otimes 1)=f\otimes 1$ for all $g\in G(R)$ and all \ks-algebras $R$. The \ks-subalgebra of $k\{X\}$ consisting of all $G$-invariant elements is denoted by $k\{X\}^G$.
An ideal $\ida\subseteq k\{X\}$ is called $G$-stable if	$g(\ida\otimes_k R)\subseteq \ida\otimes_k R$ for all $g\in G(R)$ and all \ks-algebras $R$.

\begin{lem} \label{lem: torsor and invariants}
	If $X$ is a $G$-torsor, then
	\begin{enumerate}
		\item $k\{X\}^G=k$ and
		\item the only $G$-stable ideals in $k\{X\}$ are $\{0\}$ and $k\{X\}$.
	\end{enumerate}
\end{lem}
\begin{proof}
	a) Let $f\in k\{X\}^G$. Consider $X$ as a right $G$-torsor via $X\times G\to X,\ (x,g)\mapsto g^{-1}.x$ and let $\rho\colon k\{X\}\to k\{X\}\otimes_kk\{G\}$ be the dual map of this action. Let $R$ be a \ks-algebra and $g\in G(R)$. The action of $g\in G(R)=\Hom(k\{G\},R)$ on $k\{X\}\otimes_k R$ is the $R$-linear extension of
	$k\{X\}\xrightarrow{\rho}k\{X\}\otimes_k k\{G\}\xrightarrow{\id\otimes g} k\{X\}\otimes_k R$. By choosing $R=k\{G\}$ and $g=\id\in G(R)=\Hom(k\{G\},k\{G\})$, we obtain $\rho(f)=g(f)=f\otimes 1$. Since $X$ is a right $G$-torsor, the left $k\{X\}$-linear extension
	$k\{X\}\otimes_k k\{X\}\to k\{X\}\otimes_k k\{G\}$ of $\rho$ is an isomorphism. This implies that $f\otimes 1=1\otimes f\in k\{X\}\otimes_k k\{X\}$. Thus $f\in k$.
	
	b) Let $\ida\subseteq k\{X\}$ be a $G$-stable ideal. As in a), this implies that $\rho(\ida)\subseteq \ida\otimes_k k\{G\}$. Now under the isomorphism $k\{X\}\otimes_k k\{X\}\to k\{X\}\otimes_k k\{G\}$
	the ideal $\ida\otimes_k k\{G\}\subseteq k\{X\}\otimes_k k\{G\} $ corresponds to $\ida\otimes_k k\{X\}$ and the ideal of $k\{X\}\otimes_k k\{G\}$ generated by $\rho(\ida)$ corresponds to $k\{X\}\otimes_k\ida$. Thus $k\{X\}\otimes_k\ida\subseteq\ida\otimes_k k\{X\}$. This is only possible if $\ida$ is trivial. 
\end{proof}

%
%
%
%

Let us now recall the basic definitions and results from \cite{DHW}. A \emph{$\de$-ring} is a commutative ring $R$ together with a derivation $\de\colon R\to R$. An ideal $\ida\subseteq R$ such that $\de(\ida)\subseteq \ida$ is called a \emph{$\de$-ideal}. If every $\de$-ideal of $R$ is trivial, $R$ is called \emph{$\de$-simple}.
The \emph{$\de$-constants} of $R$ are $R^{\de}=\{r\in R \mid \de(r)=0\}$.

A \emph{$\ds$-ring} is a commutative ring $R$ together with a derivation $\de\colon R\to R$ and a ring endomorphism $\s\colon R\to R$ such that $\de(\s(r))=\hslash\s(\de(r))$ for all $r\in R$ for some fixed unit $\hslash\in R^\de$.
If $R$ is a field, we speak of a $\ds$-field. There are the obvious notions of a morphism of $\ds$-rings, of $\ds$-algebras etc.

From now on let $K$ denote a $\ds$-field of characteristic zero and let $k=K^\de$ be the $\s$-field of $\de$-constants. We consider a linear differential system $\de(y)=Ay$ with a matrix $A\in K^{n\times n}$.

\begin{Def} \label{def: sPVring}
	A \emph{$\s$-Picard-Vessiot ring\footnote{This definition differs from Definition \cite[Definition 1.2]{DHW} where the condition $R^\de=k$ is dropped. Definition \ref{def: sPVring} is more convenient for us and \cite[Prop. 1.5]{DHW} shows that with our definition, $\s$-Picard-Vessiot rings correspond to $\s$-Picard-Vessiot extensions.}} for $\de(y)=A y$ is a $K$-$\ds$-algebra $R$ such that
	\begin{itemize}
		\item there exists a matrix $Y\in\GL_n(R)$ with $\de(Y)=AY$ and $R=K\{Y,1/\det(Y)\}$,
		\item $R$ is $\de$-simple,
		\item $R^\de=k$. 
	\end{itemize} 
\end{Def}
The $\s$-Galois group of $R/K$ is the $\s$-algebraic group $G$ over $k$ given by 
$$G(S)=\Aut^{\ds}(R\otimes_k S|K\otimes_k S)$$
for any \ks-algebra $S$. Here and in the sequel $R\otimes_k S$ is considered as a $\ds$-ring with $\de$ being the trivial derivation on $S$, i.e., $\de(s)=0$ for $s\in S$. The choice of $Y\in\GL_n(R)$ determines a $\s$-closed embedding of $G$ into $\GL_n$. Indeed, for every \ks-algebra $S$ and $g\in G(S)$ there exists a matrix $\phi_S(g)\in\GL_n(S)$ such that $g(Y)=Y\phi_S(g)$. Then $\phi\colon G\to \GL_n$ is a $\s$-closed embedding.
The coordinate ring of $G$ is $k\{G\}=(R\otimes_K R)^\de$ and the canonical map 
\begin{equation} \label{eqn: alg torsor isom} R\otimes_k k\{G\}\to R\otimes_K R
\end{equation}
is an isomorphism.
This isomorphism can be reinterpreted by saying that $R$ is the coordinate ring of a right $G_K$-torsor. Using the isomorphism (\ref{eqn: alg torsor isom}), it then follows as in Lemma \ref{lem: torsor and invariants} that:
\begin{lem}  \label{lem: G-stable ideal in R trivial}
	\mbox{}
	\begin{enumerate}
		\item $R^G=K$ and
	\item every $G$-stable ideal of $R$ is trivial.  \qed
	\end{enumerate}
\end{lem}
Now we are prepared to prove the main result of this section:
\begin{thm} \label{thm: isomclasses of sPVrings}
	Let $R$ be a $\s$-Picard-Vessiot ring for $\de(y)=Ay$ with $\s$-Galois group $G$. Then the set of isomorphism classes of $\s$-Picard-Vessiot rings for $\de(y)=Ay$ is in bijection with $\h^1(k,G)$.
\end{thm}
\begin{proof}
	By Theorem \ref{thm: H1 bij torsors} it suffices to establish a bijection between the isomorphism classes of $\s$-Picard-Vessiot rings for $\de(y)=Ay$ and the isomorphism classes of $G$-torsors. Indeed, we will construct an isomorphism of pointed sets such that the equivalence class of $R$ corresponds to the equivalence class of the trivial $G$-torsor.
	
	\medskip
	
	\noindent \emph{First step: From a $\s$-Picard-Vessiot ring to a $G$-torsor:} 
	
	Let $R'$ be a $\s$-Picard-Vessiot ring for $\de(y)=Ay$. For every \ks-algebra $S$ we set
	\begin{equation} \label{eqn: Isom torsor}
	X(S)=\operatorname{Isom}_{K\otimes_k S}^{\ds}(R'\otimes_k S, R\otimes_k S).
	\end{equation}
	We claim that $X$ is a \ks-variety, i.e., representable by a finitely $\s$-generated \ks-algebra.
	Set $k\{X\}=(R\otimes_K R')^\de$. According to Lemma 1.8 in \cite{DHW} the \ks-algebra $k\{X\}$ is finitely $\s$-generated and the canonical map $R\otimes_k k\{X\}\to R\otimes_K R'$ is an isomorphism (of $\ds$-rings). Moreover, as in \cite[Lemma 2.4]{DHW} (for the case $R'=R$) one shows that every $K\otimes_k S$-$\ds$-morphism from $R'\otimes_k S$ to $R\otimes_k S$ is an isomorphism. Thus we have the following chain of identifications
	\begin{align*} X(S) &=\Hom_{K\otimes_k S}^{\ds}(R'\otimes_k S, R\otimes_k S) \\
	&= \Hom_{K}^{\ds}(R', R\otimes_k S)\\
	&= \Hom_{R}^{\ds}(R\otimes_K R', R\otimes_k S) \\
	&= \Hom_{R}^{\ds}(R\otimes_k k\{X\}, R\otimes_k S) \\
	&= \Hom_{k}^{\s}(k\{X\},S),
	\end{align*}
	where the last equality holds because $(R\otimes_k S)^\de=S$. This shows that $X$ is a \ks-variety. Now $G$ acts on $X$ by composition: $g.x=g\circ x$ for $g\in G(S)$ and $x\in X(S)$ and clearly $X$ is a (left) $G$-torsor.
	
	\medskip

	\noindent \emph{Second step: From a $G$-torsor to a $\s$-Picard-Vessiot ring}
	%
	%
	
	Let $X$ be a $G$-torsor. Recall that $G$ acts (functorially from the left) on $R$ and on $k\{X\}$.
	Thus $G$ also acts (functorially from the left) on $R\otimes_k k\{X\}$. Since $G$ acts via $K$-$\ds$-automorphisms the invariants $R'=(R\otimes_k k\{X\})^G$ are a $K$-$\ds$-algebra. We will show that $R'$ is a $\s$-Picard-Vessiot ring for $\de(y)=Ay$.
	
	Let us first show that the canonical map $R\otimes_K R'\to R\otimes_k k\{X\}$ is injective. For a contradiction, assume it is not injective and let $n\geq 2$ be minimal such that there exist $n$ $K$-linearly independent elements $r_1',\ldots,r_n'\in R'$ that are $R$-linearly dependent (when viewed as elements in $R\otimes_k k\{X\}$). Let $\ida\subseteq R$ denote the non-zero ideal of all $r_1\in R$ such that there exist $r_2,\ldots,r_n\in R$ with $r_1r_1'+r_2r_2'+\ldots+r_nr_n'=0$.
	
	We will show that $\ida$ is $G$-stable. So let $S$ be a \ks-algebra and $g\in G(S)$. We have to show that $g(r_1\otimes 1)\in \ida\otimes_k S$ for all $r_1 \in \ida$.
	As $g(r_i'\otimes 1)=r_i'\otimes 1$ for $i=1,\ldots n$ we have
	$g(r_1\otimes 1)(r_1'\otimes 1)+\ldots+g(r_n\otimes 1)(r_n'\otimes 1)=0$.
	
	Let $(f_j)_{j\in J}$ be a $k$-basis of $S$ with $1\in J$ and $f_1=1$. Since $G$ acts on $R$, we have $g(r_i\otimes 1)\in R\otimes_k k\otimes_k 
	S\subseteq R\otimes_k k\{X\}\otimes_k S$ for $i=1,\ldots,n$. So we can write
	$$g(r_i\otimes 1)=\sum_j a_{ij}\otimes 1\otimes f_j $$
	with $a_{ij}\in R$ for $i=1,\ldots,n$.
	Then
	$$\sum_i\sum_j a_{ij}r_i'\otimes f_j=0$$ and therefore
	$\sum_ia_{ij}r_i'=0$ for all $j\in J$. So $a_{1j}\in\ida$ for all $j\in J$. This shows that $g(r_1\otimes 1)=\sum_j a_{1j}\otimes f_j\in\ida\otimes_k S$.
%
	Thus $\ida$ is $G$-stable. Since $R$ itself is the only non-zero $G$-stable ideal in $R$ (Lemma \ref{lem: G-stable ideal in R trivial}) we see that $1\in\ida$. So we may assume that $r_1=1$. Then
	$g(r_1\otimes 1)=1\otimes 1\otimes  1=a_{11}\otimes 1 \otimes f_1$, i.e., $a_{11}=1$ and $a_{1j}=0$ for $j\neq 1$. As $\sum_ia_{ij}r_i'=0$, it follows from the minimality of $n$ that $a_{ij}=0$ for all $i=1,\ldots,n$ and $j\neq 1$. Thus $g(r_i\otimes 1)=a_{i1}\otimes 1\otimes 1$ for $i=1,\ldots,n$. 
	But then $\sum a_{i1}r'_i=0$ and so the minimality of $n$ and the fact that $a_{11}=1$ implies that $a_{i1}=r_i$ for $i=1,\ldots,n$.
Thus $g(r_i\otimes 1)=r_i\otimes 1$ for $i=1,\ldots,n$. Hence $r_i\in R^G=K$ for $i=1,\ldots,n$; a contradiction.

	\medskip
	
	This shows that $R\otimes_K R'\to R\otimes_k k\{X\}$ is injective. Our next concern is to find a $Y'\in\GL_n(R')$ with $\de(Y')=AY'$.
	As explained after Definition \ref{def: sPVring}, we can consider $G$ as a $\s$-closed subgroup of $\GL_n$. So by Lemma \ref{lem: embed into Gln}, we can assume without loss of generality that $X$ is a $\s$-closed $\s$-subvariety of $\GL_n$ and that the (left) action of $G$ on $X$ is given by matrix multiplication. Let $k\{\GL_n\}=k\{T,1/\det(T)\}$ be the (difference) coordinate ring of $\GL_n$ and let $Z$ denote the image of $T$ under the morphism $k\{\GL_n\}\to k\{X\}$ corresponding to the inclusion $X\subseteq \GL_n$. For any \ks-algebra $S$ and $g\in G(S)$ we have $g(Z)=g^{-1}Z$ and therefore
	$$ g(Y\otimes Z)=Yg\otimes g^{-1}Z=Y\otimes Z.$$
	Thus $Y'=Y\otimes Z\in\GL_n(R')$. Since the entries of $Z$ are $\de$-constants it is clear that $\de(Y')=AY'$. Moreover, as $1\otimes Z=(Y\otimes 1)^{-1}Y'$, the map $R\otimes_K K\{Y',1/\det(Y')\}\to R\otimes_k k\{X\}$ is surjective.
	Since $R\otimes_K K\{Y',1/\det(Y')\}\subseteq R\otimes_K R'$ and $R\otimes_K R'$ injects into $R\otimes_k k\{X\}$, this implies that $ R'=K\{Y',1/\det(Y')\}$ and
	$R\otimes_K R'\cong R\otimes_k k\{X\}$.

\medskip
	
	Let us next show that $R'$ is $\de$-simple. Let $\ida$ be a proper differential ideal in $R'$ and let $\widetilde{\ida}$ be the ideal generated by $\ida$ in $R\otimes_k k\{X\}$. Since $R\otimes_K \ida\subseteq R\otimes_K R'$ is a proper $\de$-ideal, also $\widetilde{\ida}\subseteq R\otimes_k k\{X\}$ is a proper differential ideal. By \cite[Lemma 2.3]{DHW} the $\de$-ideal $\widetilde{\ida}$ is generated by $\widetilde{\ida}\cap k\{X\}$. Moreover, since the elements of $\ida$ are fixed by the $G$-action, it is clear that $\widetilde{\ida}$ is $G$-stable, so also $\widetilde{\ida}\cap k\{X\}\subseteq k\{X\}$ is $G$-stable. But according to Lemma \ref{lem: torsor and invariants}, all $G$-stable ideals in $k\{X\}$ are trivial. Thus $\widetilde{\ida}\cap k\{X\}=\{0\}$ and therefore also $\widetilde{\ida}$ and $\ida$ are the zero ideal.
	
	Thus $R'$ is $\de$-simple. To finish the proof that $R'$ is a $\s$-Picard-Vessiot ring it suffices to show that $R'^\de=k$. But as $(R\otimes_k k\{X\})^\de=k\{X\}$, we see that
	$$R'^\de=k\{X\}^G=k$$
	by Lemma \ref{lem: torsor and invariants}. Thus $R'$ is a $\s$-Picard-Vessiot ring for $\de(y)=Ay$.
	
	\medskip
	
	It remains to see that the above two constructions induce bijections on the isomorphism classes. If we start with a $\s$-Picard-Vessiot ring $R'$ and define $X$ as in step one, then $R\otimes_k k\{X\}\cong R\otimes_K R'$. We will show that this isomorphism is compatible with the functorial $G$-action if we let $G$ act on $R'$ trivially: If $Y\in\GL_n(R)$ with $\de(Y)=AY$ and $Y'\in\GL_n(R')$ with $\de(Y')=AY'$, then $k\{X\}=k\{Z,1/\det(Z)\}$ where $Z=Y^{-1}\otimes Y'\in\GL_n(R\otimes_K R')$ (\cite[Lemma 1.8]{DHW}). Thus we can consider $X$ as a $\s$-closed $\s$-subvariety of $\GL_n$. For a \ks-algebra $S$, an element $x\in X(S)\subseteq\GL_n(S)$ corresponds to the isomorphism $R'\otimes_k S\to R\otimes_k S$ determined by $Y'\mapsto Yx$. Moreover, the composition with an element $g\in G(S)$ corresponds to the isomorphism determined by $Y'\mapsto Ygx$. Thus $G$ acts on $X$ by matrix multiplication from the left. So on the coordinate ring $k\{X\}$, the action of $G$ is given by $g(Z)=g^{-1}Z=g^{-1}(Y^{-1}\otimes Y')$. This shows that the isomorphism $R\otimes_k k\{X\}\cong R\otimes_K R'$ is compatible with the $G$-action if $g(Y')=Y'$. Since $(R\otimes_K R')^G\cong R'$ it follows that also $(R\otimes_k k\{X\})^G\cong R'$.
	
	Conversely, if we start with a $G$-torsor $X$ and define $R'=(R\otimes_k k\{X\})^G$, then $R\otimes_K R'\to R\otimes_k k\{X\}$ is an isomorphism and therefore $(R\otimes_K R')^\de\cong k\{X\}$. We claim that the induced isomorphism of $\s$-varieties $X\to\widetilde{X}$ is an isomorphism of $G$-torsors: Let $Y\in\GL_n(R)$ with $\de(Y)=AY$ and define $Z$ and $Y'=Y\otimes Z\in\GL_n(R')$ as in the second step above. Then  $(R\otimes_K R')^\de=k\{\widetilde{Z},1/\det(\widetilde{Z})\}$ where $\widetilde{Z}=Y^{-1}\otimes Y'$. As the isomorphism $(R\otimes_K R')^\de\cong k\{X\}$ maps $\widetilde{Z}$ to $Z$, we see that $\widetilde{X}$ and $X$ agree as $\s$-closed $\s$-subvarieties of $\GL_n$. As $G$ is acting on both by matrix multiplication from the left, we conclude that $X\to\widetilde{X}$ is an isomorphism of $G$-torsors.
\end{proof}

Let $R'$ and $R$ be $\s$-Picard-Vessiot rings for $\de(y)=Ay$ and let $B$ be a \ks-algebra. We say that $R$ and $R'$ are isomorphic over $B$ if there exists an isomorphism of $K\otimes_k B$-$\ds$-algebras between $R'\otimes_k B$ and $R\otimes_k B$. For the torsor $X$ defined by Equation (\ref{eqn: Isom torsor}) we habe $X(B)\neq\emptyset$ if and only if $R$ and $R'$ are isomorphic over $B$. So we obtain 
from Theorem \ref{Klassifikation}:
%
\begin{cor}
	Let $R$ be a $\s$-Picard-Vessiot ring for $\de(y)=Ay$ with $\s$-Galois group $G$. Then for a \ks-algebra $B$, the set of isomorphism classes of $\s$-Picard-Vessiot rings for $\de(y)=Ay$ that are isomorphic to $R$ over $B$ is in bijection with $\h^1(B/k,G)$.
\end{cor}

\bibliographystyle{amsalpha}
 \bibliography{references}
\end{document}